\documentclass[a4paper,10pt,reqno]{amsart}
\usepackage[T1]{fontenc}
\usepackage[utf8]{inputenc}
\usepackage{xcolor}
\usepackage{amssymb}
\usepackage{eucal}
\usepackage{a4wide}
\usepackage{graphicx}
\usepackage[all]{xy}
\usepackage{amsmath,mathrsfs,mathtools}
\usepackage{wasysym}
\usepackage{dsfont}
\usepackage{enumitem}
\usepackage{float}
\usepackage{fancyhdr}
\usepackage{comment}
\usepackage[labelfont=bf]{caption}
\usepackage{algorithmicx}
\usepackage{algorithm}
\usepackage{algpseudocode}
\usepackage{xpatch}
\usepackage{hyperref}
\xpatchcmd{\paragraph}{\normalfont}{{\normalfont\itshape}}{}{}

\newtheorem{assumption}{Assumption}[section]
\newtheorem{theorem}{Theorem}[section]
\newtheorem{lemma}{Lemma}[section]
\newtheorem{remark}{Remark}[section]
\theoremstyle{definition}
\newtheorem{definition}{Definition}[section]
\graphicspath{ {./Pictures/} }
\newcommand{\lp}{\left(}
\newcommand{\rp}{\right)}
\newcommand{\m}{\mathfrak{m}}
\newcommand{\der}{\partial}
\newcommand{\R}{\mathbb{R}}      

\newcommand{\N}{\mathbb{N}}      

\newcommand{\Flder}{\rightarrow}
\newcommand{\dd}{\mathrm{d}}     
\newcommand{\Lc}{\mathcal{Lc}^{\text{con}}}
\newcommand{\Lf}{\mathcal{L}^{\text{fra}}}
\renewcommand{\L}{\mathcal{L}}
\DeclareMathOperator{\diag}{diag}
\newcommand{\dis}{\text{dis}}

\begin{document}
\title{On Runge-Kutta convolution quadrature based fractional variational integrators}

\begin{abstract}
Lagrangian systems subject to fractional damping can be incorporated into a 
  variational framework by doubling the state 
  variables and introducing fractional derivatives. 
Fractional variational integrators based on backward-differentiation convolution quadrature (BDFCQ),
  combined with higher-order Galerkin methods, saturate at second-order accuracy because the multistep structure of BDFCQ does not take into account the internal stages of the Galerkin discretization. The main objective of this paper is to develop fractional variational integrators (FVIs) by combining Runge-Kutta convolution quadrature (RKCQ) for the approximation of fractional derivatives with higher-order Galerkin methods. The RKCQ approach is naturally compatible with such stage-based discretizations and is therefore better suited for the
  construction of higher-order schemes.  We are particularly interested in the CQ based on Lobatto~IIIC. Preservation properties such as energy decay, as well as convergence properties, are investigated numerically and proved for
  second-order schemes. The presented schemes reach 2nd, 4th and 6th order of
  accuracy. A brief discussion on the midpoint fractional integrator is also
  included.
\end{abstract}



\author{Khaled Hariz Belgacem}
\email{hariz@math.upb.de}
\address{Department of Mathematics, University of Paderborn, Warburger Straße 100, 33098 Paderborn, Germany}

\author{Sina Ober-Bl\"obaum}
\email{sinaober@math.upb.de}
\address{Department of Mathematics, University of Paderborn, Warburger Straße 100, 33098 Paderborn, Germany}

\email{fjimenez@ind.uned.es}
\author{Fernando Jim\'enez}
\address{Departamento de Matemática Aplicada I, ETSII, Universidad Nacional de Educación a Distancia (UNED)
c. Juan del Rosal 12, 28040, Madrid, Spain.}

\maketitle

\section{Introduction}
Differential equations are extensively used to model many phenomena in science and engineering;  however, they are often not capable of describing systems having long-term memory property due to their local nature. For instance, phenomena like anomalous diffusion, deformations, viscoelastic materials~\cite{Bagley}, biological processes~\cite{Magin1, Magin2}, and other applications in different fields~\cite{Bonilla, Hilfer} require models that incorporate historical effects. In this context, fractional differential equations, which involve fractional derivatives, provide a powerful extension. Their nonlocal nature allows the current behaviour of a system to depend not only on the present but also on its entire past history, leading to more accurate and realistic representations of complex phenomena. 
We refer to~\cite{Oldham,TheBook,Igor,TheBook2} for a comprehensive overview of the fractional calculus theory, to~\cite{Tarasov} for physical applications, to \cite{chala} for applications in mechanics such as nonlocal elasticity, viscoelasticity, and 
heat conduction, and to~\cite{Sawar_Hussain_Ayaz_2025} for a recent development in cryptography.

In recent years, fractional calculus has also emerged as a powerful tool in the study of dynamical systems. Its applications extend across a wide range of models, including dissipative and conservative systems as well as constrained dynamical systems~\cite{Riewe,Riewe2,Klimek,Agra,Baleanu}. Of particular interest in this work are fractional variational problems, which extend the classical variational framework to the fractional setting. Such problems naturally give rise to fractional equations of motion, generalizing their classical counterparts.

Classical variational principles fail to capture dissipative systems~\cite{Bauer}; fractional calculus provides a natural extension of the variational framework, enabling the dissipation to be fully treated within the framework of variational principles.
This topic has been discussed by several authors~\cite{Riewe, Cresson2,Cresson3, JiOb1,JiOb2}. In particular, Lagrangian systems subject to linear  damping with constant coefficients have been discussed in~\cite{JiOb2} by means of the restricted Hamilton principle. This approach shows that the dynamical equations are sufficient but not necessary conditions for characterizing the extremals of the action. More recently,~\cite{HaJiOb} proposed a generalization employing BDFCQ for
the fractional derivatives, combined with Galerkin variational methods for the conservative part.

At the discrete level, the restricted Hamilton principle gives rise to Fractional Variational Integrators by employing two different approximations: one for the conservative part of the  Lagrangian and one for the fractional part. A well-known discretization for the conservative part can be found in~\cite{MaWe,LuGeWa}. For the fractional part, several approximations are available~\cite{Lidia2015,BoXuHu,HandbookNum,Kaibook2016}, but the chosen method must ensure that essential properties are preserved in the discrete setting, as in the continuous one. In this context, FVIs have been derived in~\cite{JiOb2} using the Grünwald-Letnikov approximation, yielding first-order accurate schemes. More recently,~\cite{HaJiOb} proposed a generalization by employing BDFCQ~\cite{Lubich2} for the fractional derivatives, combined with Galerkin variational methods for the conservative part (see~\S\ref{HO-Action}). Although the BDFCQ can in principle attain orders of accuracy up to six with suitable correction terms, in the fractional variational setting the overall accuracy of the resulting FVIs is limited to second order, because it neglects the inner nodes used in higher-order approximations of the conservative action; see~\cite{HaJiOb} for further details.

To overcome this limitation, one requires provably higher-order fractional variational integrators whose discretization of the fractional part is structurally compatible with the internal-stage structure of Galerkin approximations of the conservative action. Constructing such integrators is the problem addressed in the present work.

The rest of the paper is organized as follows. Section~\ref{Disc_Lag_Mech} gives a brief summary of the Lagrangian formalism and of higher-order Galerkin variational integrators  used
throughout. Section~\ref{FracIntegrals} recalls fractional operators and their approximations by
RKCQ. The core of the paper is in Section~\ref{sec:FVI}. After recalling the continuous restricted Hamilton principle in \S\ref{sec:restrectedCOV}, we show in~\S\ref{subsec:Discrete_Setting} how RKCQ (based on Lobatto~IIIC) can be combined with the discrete restricted Hamilton principle to derive the fractional discrete Euler-Lagrange (FDEL) equations (Theorem~\ref{thm:maintheorem}, Algorithm~\ref{alg:FractionalAlgorithm}). Section~\ref{sec:numericalresult} is devoted to the convergence analysis and numerical experiments for the proposed schemes. There, we consider two 
  models: the two-dimensional forced oscillator and the fractional Bagley-Torvik equation. For the 2-stage Lobatto~IIIC scheme applied to a quadratic 
  Lagrangian, the second-order accuracy is established analytically in
  Theorem~\ref{thm:theoretical_convergence}, while higher-order convergence    and the energy-decay behaviour of the $r$-stage Lobatto~IIIC schemes ($r=2,3,4$) are illustrated numerically in the same section. A brief discussion of the midpoint scheme (MIDCQ) is also included. Section~\ref{conclusion} concludes the paper.
\section{Discrete Lagrangian Mechanics}\label{Disc_Lag_Mech}

In this section, we give a brief summary of the Lagrangian mechanical system from both continuous and  discrete points of view. 

\subsection{Lagrangian Hamilton's Principle}
Consider a mechanical system with a $d$-dimensional configuration manifold $Q$, whose motion is described by a trajectory in the tangent bundle $TQ$ associated to $Q$, parameterized by positions 
$q(t)\in Q$ and   velocities  $\dot q(t)\in T_{q(t)}Q$. The Lagrangian system is characterized by a $C^2$-Lagrangian $L:TQ\Flder\R$, commonly given by  $L(q,\dot q)=T(q,\dot q)-V(q)$, where  $T$ and $V$ denote the kinetic and potential energies, respectively. Consider 
the action $\mathcal{L}:C^2([0,T],Q)\Flder\R$ of a mechanical system, given by
\begin{equation}\label{ContAc}
\mathcal{L}(q)=\int_0^TL(q(t),\dot q(t))\,\dd t.
\end{equation}
In Lagrangian mechanics, the trajectory is characterized by Hamilton's principle.  That is, the trajectory $q$, with fixed points $q(0)$ and $q(T)$, is the curve along which the action is stationary, namely
\[
\delta\mathcal{L}(q)=0.
\]
for all variations that satisfy $\delta q(0)=\delta q(T)=0$. The associated necessary and sufficient conditions are the Euler-Lagrange  equations
\begin{equation}\label{EL}
\frac{\dd}{\dd t}\lp\frac{\der L}{\der\dot q}(q(t),\dot q(t))\rp-\frac{\der L}{\der q}(q(t),\dot q(t))=0.
\end{equation}
These are the equations of motion for a conservative system. See~\cite{AbMa,MaWe,MarRat} for more details.

\subsection{Discrete Lagrangian Hamilton's Principle}\label{Disc}
Discrete mechanics is essentially a discrete formulation of geometric mechanics that preserves its underlying structure within a numerical framework, leading to what are also called \textit{structure-preserving methods}.  Variational integrators (VIs) are a discrete formulation of the Euler-Lagrange equation~\eqref{EL} and belong to a family of structure-preserving methods. Unlike standard discretizations of differential equations,  the construction of VIs starts by discretizing  the state space $TQ$ and the action $\mathcal{L}$~\cite{MaWe}. More concretely, consider a time discretization $\left\{t_k\right\}_{k=0}^N$ of the interval $[0,T]$ with $t_k=kh$, where $h=T/N$ is the time step. One approximates the continuous trajectory $q(t)$ by a discrete sequence $q_\dis^{}:=\left\{ q_k\right\}_{0:N}\in Q^{N+1}$ at times $t_k$, i.e.,~$q_k\approx q(t_k)$. The discrete Lagrangian $L_\dis $ has to be defined in such a way as to approximate the action over a short time interval.  Thus, define $L_\dis:Q\times Q\Flder\R$ based on two neighboring configurations $q_k$ and $q_{k+1}$ as
\begin{equation}\label{TwoPointsApprox}
L_\dis (q_k,q_{k+1})\approx \int_{t_k}^{t_{k+1}}L(q(t),\dot q(t))\,\dd t.
\end{equation}
Furthermore, the discrete action sum $\mathcal{L}_\dis :Q^{N+1}\Flder\R$ approximates the continuous one~\eqref{ContAc} as
\[
\mathcal{L}_\dis (q_\dis^{})=\sum_{k=0}^{N-1}L_\dis (q_k,q_{k+1}).
\]
The discrete variational principle is designed to mirror the continuous one. Namely, the trajectory $q_\dis =\{q_k\}_{k=0}^N$, with fixed points $q_0$ and $q_N$, is determined requiring that the action sum be stationary
\[
\delta \mathcal{L}_\dis (q_\dis )=0,
\]
for arbitrary variations $\delta q_k$ satisfying $\delta q_0 = \delta q_N =
   0$, leading to the discrete Euler-Lagrange (DEL) equation 
\begin{equation}\label{DEL}
D_2L_\dis (q_{k-1},q_{k})+D_1L_\dis (q_k,q_{k+1})=0,\qquad k=1,\ldots, N-1,
\end{equation}
where $D_i$ refers to the partial derivative with respect to the $i$-th argument. This scheme is called a variational integrator and its flow is defined implicitly by the map  $\Phi_h:Q\times Q\Flder Q\times Q$, $(q_{k-1},q_k)\mapsto (q_k,q_{k+1})$. 
The resulting integrator~\eqref{DEL} belongs to the class of structure-preserving methods~\cite{LuGeWa,Raff2023,SerFer2016}, meaning that it preserves the symplectic structure exactly, conserves momentum maps 
  associated with continuous symmetries, and exhibits good long-time energy    behaviour. Despite this advantage, this construction based on only two configurations $q_k$ and $q_{k+1}$ yields at most a second-order accurate scheme. The purpose of the following paragraph is to introduce a family of variational methods of arbitrary order. 

\subsection{Galerkin Variational Methods}\label{HO-Action}
We now consider another class of methods to generate variational integrators, known as Galerkin methods. This approach involves locally interpolating the trajectories on a finite-dimensional basis and evaluating the action integral using a suitable quadrature method~\cite{HaLe13,MaWe,SinaSaake}.

\paragraph{(1) Trajectory space} We replace the infinite-dimensional function space
$$\mathcal{C}([t_k,t_{k+1}],Q)=\left\{ q:[t_k,t_{k+1}]\Flder Q\,|\, q(t_k)=q_k,\,q(t_{k+1})=q_{k+1} \right\},$$
by a finite-dimensional one  $\mathcal{C}^s([t_k,t_{k+1}],Q)\subset \mathcal{C}([t_k,t_{k+1}],Q)$, where $\mathcal{C}^s([t_k,t_{k+1}],Q)$ denotes the space of polynomials of degree $s$.  Given $s+1$ control points $0=d_0<d_1<\cdots<d_{s-1}<d_s=1$ and $s+1$ configurations $q_k=(q_k^0,q_k^1,\ldots,q_k^{s-1},q_k^s)$, which satisfy $q_k^0=q_k$ and $q_k^{s}=q_{k+1}$, there exists a basis $\{\ell_\nu\}_{\nu=0}^s$, e.g.,~the Lagrange polynomials, for which an element $q_\dis^{} \in \mathcal{C}^s([t_k,t_{k+1}],Q)$ can be uniquely written as   a  linear combination of $\{\ell_\nu\}_{\nu=0}^s$ 
\begin{equation}\label{Polynomials}
q(t;k):=\sum_{\nu=0}^sq_k^{\nu}\,\ell_{\nu}\left(\frac{t-t_k}{h}\right),
\end{equation}
 such that $\ell_{\nu}(d_i)= \delta_{\nu i}$ and therefore $q_\dis^{} (d_\nu h;k)=q_k^\nu$ according to~\eqref{Polynomials}.  A trajectory $q_\dis^{}  : [0, T ] \to Q$, defined over the entire interval, is expressed on each sub-interval using the local interpolation $q_\dis^{}$. This gives, for all $t \in [t_k , t_{k+1}]$
\[q_\dis^{}|_{[t_k , t_{k+1}]} (t) = q_\dis^{} (t;k).\]
To ensure the continuity of $q_\dis^{}$ at the main nodes, we write for all $0 < k < N$ the condition
\[q_\dis^{} (t_k;k) = q_\dis^{} (t_k;k-1)=q_k.\]

\paragraph{(2) Quadrature for the action integral} We first replace the curve $q(t)$ and the velocity $\dot q(t)$ by their polynomial counterparts $q_\dis^{}(t;k)$, $\dot q_\dis^{}(t;k)$ in the action as
\begin{equation}\label{QuadforL_d}
    \int_{t_k}^{t_{k+1}}L(q_\dis^{}(t;k),\dot q_\dis^{}(t;k))\,\dd t,\quad k=0,\ldots, N-1.
\end{equation}
Applying a quadrature rule $(b_i,c_i)_{i=1}^r$ with $c_i\in[0,1]$ to~\eqref{QuadforL_d}, we obtain
\begin{equation}\label{DiscLag}
L_\dis(\{q_k\})\equiv L_\dis(q_k^0, \ldots,q_k^s):=h\sum_{i=1}^rb_iL(q_\dis^{}(c_i\,h;k),\dot q_\dis^{}(c_i\,h;k))\approx \int_{t_k}^{t_{k+1}}L(q(t),\dot q(t))\,\dd t.
\end{equation}
Following~\S\ref{Disc}, the action sum can be defined from~\eqref{DiscLag} by
\[
\mathcal{L}_\dis(q_\dis^{})=\sum_{k=0}^{N-1}L_\dis(\{q_k\}),
\]
and hence the discrete Hamilton principle leads to the discrete Euler-Lagrange equations
\[
\begin{split}
D_{s+1}L_\dis(q_{k-1}^0,\ldots,q_{k-1}^s)+D_1L_\dis(q_{k}^0,\ldots,q_{k}^s)=0,\\
D_iL_\dis(q_{k}^0,\ldots,q_{k}^s)=0,\quad  \forall\,i=2,\ldots,s;
\end{split}
\]
for $k=1,\ldots,N-1$ where $D_iL_\dis(q_k^{0},\ldots,q_k^{s}):=\frac{\partial L_\dis(\{q_k\})}{\partial q_k^{i-1}}$. See~\cite{MaWe, SinaSaake} for further details.

Variational integrators have been widely used for conservative systems~\cite{HaLe13,Leok2011,SinaSaake,Ca13}, nonconservative systems~\cite{SinaVerm,ShPaWo}, optimal control problems~\cite{Ca14}, and multirate systems~\cite{SinaWeGaLe,WeSinaLe,LeSina}. The discrete variational principle automatically transfers the accuracy of 
  the discrete Lagrangian to the resulting DEL equation by means of what is called {\it local variational order}~\cite{MaWe}, that is,  \textit{to obtain a geometric integrator of order $r$, it suffices to construct an  approximation of the action integral of order $r+1$}. Depending on the degree of the interpolating polynomial and choice of the quadrature rule, one obtains  VIs of different computational complexity and accuracy.  Galerkin variational methods with polynomial interpolation of degree $s$ and quadrature of order $s$ were shown to converge with order at least $s$~\cite{HaLe13} and the optimal order is generally unknown. Numerical studies~\cite{SinaSaake} indicate convergence of order $\min(2s, u)$ for Lobatto and Gauss quadrature rules where $u$ is the order of the  quadrature rule.  This was also shown analytically in~\cite{SinaVerm}.

\section{Fractional Operators and their RKCQ Approximations}\label{FracIntegrals}

\subsection{Fractional Operators}
In this section, we collect the definitions of fractional integrals and derivatives and their main properties without proofs. For more details, see~\cite{TheBook2,TheBook}.

\begin{definition}[Fractional integrals]
Let $f:[0,T]\Flder\R$ be a continuous and integrable function. The left and right Riemann–Liouville fractional integrals of order $\alpha\in \R^+$ are defined respectively in the following way:   
\begin{subequations}\label{RLInt}
\begin{align}
J^{\alpha}_{-}f(t)&=\frac{1}{\Gamma(\alpha)}\int_0^t(t-\tau)^{\alpha-1}f(\tau)\,\dd\tau, \quad t\in (0,T],\label{RLInt:a}\\
J^{\alpha}_{+}f(t)&=\frac{1}{\Gamma(\alpha)}\int_t^T(\tau-t)^{\alpha-1}f(\tau)\,\dd\tau,\quad t\in [0,T), \label{RLInt:b}
\end{align}
\end{subequations}
where $\Gamma$ is the Euler Gamma function and for $\alpha=0$,  we set $J^{0}_{-}f=J^{0}_{+}f=f$.
\end{definition}
The fractional integrals enjoy two significant properties which are crucial in applications,   the {\it integration by parts} and  the {\it semigroup property}\footnote{Formula~\eqref{IntegrationByParts} holds true  for $\alpha>0$, $f \in  L^p (a, b),\, g\in L^q (a, b)$ with $p \geq 1,\, q \geq 1$ and
$\frac{1}{p} + \frac{1}{q} \leq 1 + \alpha$, while formula~\eqref{Aditive} holds almost everywhere in $[a,b]$ for $\alpha,\beta>0$ and $f \in  L^p (a, b)$ with
$1\leq p \leq\infty $.}~\cite[Lemmas 2.3 and 2.7]{TheBook2}
\begin{subequations}\label{FracIntProperties}
\begin{align}
\int_0^Tf(t)\left(J^{\alpha}_{\sigma}g\right)(t)\, \dd t&=\int_0^T g(t) \left(J^{\alpha}_{-\sigma}f\right)(t)\, \dd t,\quad  \sigma\in\left\{-,+\right\}, \label{IntegrationByParts}\\
J_{\sigma}^{\alpha}J_{\sigma}^{\alpha}f&=J_{\sigma}^{2\alpha}f,\quad \alpha>0. \label{Aditive}
\end{align}
\end{subequations}

\begin{definition}[Fractional derivatives]
The left and right Riemann-Liouville fractional derivatives of order $\alpha\in \R^+$  are defined respectively in the following way: 
\begin{subequations}\label{RLDer}
\begin{align}
D^{\alpha}_{-}f(t)&=\frac{1}{\Gamma(n-\alpha)}\left(\frac{\dd}{\dd t}\right)^n\int_0^t(t-\tau)^{n-\alpha-1}f(\tau)\,\dd\tau,\quad\,t\in (0,T],\label{RLDer:a}\\
D^{\alpha}_{+}f(t)&=\frac{1}{\Gamma(n-\alpha)}\left(-\frac{\dd}{\dd t}\right)^n\int_t^T(\tau-t)^{n-\alpha-1}f(\tau)\,\dd\tau,\quad t\in [0,T), \label{RLDer:b}
\end{align}
\end{subequations}
with $n \in \N$, $n -1\leq \alpha < n$.
\end{definition}
In particular, for $\alpha=n\in\N$, we recover the usual derivatives, namely
\[D^{n}_{-}f(t)=f^{(n)}(t),\qquad D^{n}_{+}f(t)=(-1)^nf^{(n)}(t).\]
In variational problems, two important properties for fractional derivatives\footnote{Formula~\eqref{IntegrationByParts2} holds true  for $\alpha>0$, $f \in  I_{-\sigma}^\alpha(L^p),\, g\in I_{\sigma}^\alpha(L^q)$ with $p \geq 1,\, q \geq 1$, and
$\frac{1}{p} + \frac{1}{q} \leq 1 + \alpha$, and \eqref{Aditive2} holds for $f\in I_\sigma^{2\alpha}(L^p)$, see~\cite{TheBook} for the definition of $I_-^\alpha(L^p)$ and $I_+^\alpha(L^p)$} play a significant role:
\begin{subequations}\label{FracDerProperties}
\begin{align}
\int_0^Tf(t)\left(D^{\alpha}_{\sigma}g\right)(t)\, \dd t&=\int_0^T g(t) \left(D^{\alpha}_{-\sigma}f\right)(t)\, \dd t,\quad  \sigma\in\left\{-,+\right\}, \label{IntegrationByParts2}\\
D_{\sigma}^{\alpha}D_{\sigma}^{\alpha}f&=D_{\sigma}^{2\alpha}f,\quad \alpha>0,\label{Aditive2}
\end{align}
\end{subequations}
It is important to note that, unlike fractional integrals which satisfy the semigroup property, fractional derivatives do not in general exhibit such a property unless additional conditions are imposed, see~\cite[Theorem 2.13]{Kai} and~\cite{TheBook} for further details and proofs.

\begin{remark}
Following~\cite[Lemma 2.2]{TheBook2}, the existence of \eqref{RLDer} requires $f\in AC^{n}([0,T])$, where $AC^{n}$ denotes the space of $n$-times absolutely continuous functions; see~\cite[Equation (1.1.7)]{TheBook2} or~\cite[Definition 1.3]{TheBook}. In this case, the left Riemann-Liouville fractional derivative can be rewritten  as
\begin{equation}\label{caputo}
 D^{\alpha}_{-}f(t)=\sum_{k=0}^{n-1}\frac{f^{(k)}(0)}{\Gamma(1+k-\alpha)}t^{k-\alpha}+\frac{1}{\Gamma(n-\alpha)}\int_0^t(t-\tau)^{n-\alpha-1}f^{(n)}(\tau)\,\dd\tau. 
\end{equation}
The second term on the right hand-side of \eqref{caputo} is usually referred to as the \textnormal{left Caputo fractional derivative}. On the other hand, the property~\eqref{Aditive2} is guaranteed, for instance, when $f\in AC^{n}([0,T])$ with $f^{(k)}(0) = 0,\ k = 0, \ldots , n - 1$, and in this case, the Riemann-Liouville and the Caputo derivatives are identical.
\end{remark}

\subsection{Runge-Kutta Convolution Quadratures}\label{sec:RKCQ}
The theory of convolution quadrature methods (CQ) was developed in~\cite{Lubich1,Lubich2,Lubich3}. It provides a systematic framework to compute  the convolution integral 
\begin{equation}\label{ContConv}
\left(\kappa\ast f\right)(t):=\int_0^t\kappa(t-\tau)f(\tau)\,\dd\tau, \quad t>0
\end{equation}
efficiently without evaluating the convolution kernel $\kappa$, but instead through its Laplace transform $K(s)$ by combining quadrature rules with linear multistep methods. A particular case of convolution is the {\it retarded} fractional integral~\eqref{RLInt:a}, where the  kernel $\kappa$ is given by
\begin{equation}\label{ConvKer}
\kappa(t)=\kappa^{(\alpha)}(t)=\frac{t^{\alpha-1}}{\Gamma(\alpha)}.
\end{equation}
with the corresponding Laplace transform of the form
\begin{equation}\label{LaplaceTrans}
K^{(\alpha)}(s):=\mathscr{L}(\kappa)(s)=\int_0^{\infty}\frac{t^{\alpha-1}}{\Gamma(\alpha)}e^{-st}\, \dd t=\frac{1}{s^\alpha}.
\end{equation}

Subsequent extensions~\cite{LuOs, BaLu,BaLo} have applied CQ with Runge-Kutta methods to obtain higher-order accuracy; this is the approach we adopt in in \S\ref{sec:FVI}. 

We first begin by recalling the notion of the Runge–Kutta (RK) notation and then describe the associated convolution quadrature. We refer to~\cite{HaWa,Raff2023} for the following.  An $r$-stage RK method is described by the coefficient matrix $A=(a_{ij})_{i,j=1}^r\in\R^{r\times r}$, the vector of weights $\mathbf{b}=(b_1,\ldots,b_r)^\top\in\R^r$ and the vector of abscissae $\mathbf{c}=(c_1,\ldots,c_r)^\top\in[0,1]^r$, or by means of the standard Butcher tableau
$$\setlength{\tabcolsep}{20pt}
\renewcommand{\arraystretch}{1.5}\begin{array}{c|c}
\mathbf{c}   &  A  \\ \hline
     &  \mathbf{b} ^\top 
\end{array}.$$
Let the time step be $h = T /N$ and the grid points $t_j = jh$. An $r$-stage RK method applied to ordinary differential equations
\[u'(t)=F(t,u(t)),\quad u(0)=u_0\]
is given by 
\begin{subequations}\label{eq:RKmethod}
\begin{align}
U_{ki}&= u_k +h\sum_{j=1}^r a_{ij}F(t_k+c_jh,U_{kj}),\quad i=1,\ldots,r\\
u_{k+1}&= u_k +h\sum_{j=1}^r b_{j}F(t_k+c_jh,U_{kj}).
\end{align}
\end{subequations}
The method~\eqref{eq:RKmethod} has \textit{classical order} $p\geq 1$ and \textit{stage order} $q \leq p$, respectively (see~\cite{HaWa}) if, for sufficiently smooth $F$, and assuming in~\eqref{eq:RKmethod} that  $u_k:=u(t_k)$, the following hold
\[\big\|u_{k+1}-u(t_{k+1})\big\|=\mathcal{O}(h^{p+1}),
\]
and 
\[\big\|U_{ki}-u({t_k+c_ih})\big\|=\mathcal{O}(h^{q+1}),\quad i=1,\ldots,r.\]
Let $I$ be the identity matrix in $\R^{r\times r}$ and $\mathds{1}=(1,\ldots,1)^\top$. The corresponding stability function is given by
\begin{equation}\label{StabFunc}
R(z)=1+z \mathbf{b} ^\top (I-z  A )^{-1}\mathds{1}.
\end{equation}

\begin{assumption}\label{assumption}
From now on, we make the following assumption:
\begin{enumerate}[leftmargin=2.5em]
\item The matrix $A$ is invertible.
\item  $R(iy)\leq 1$, for all real $y$.
\item $R(z)$ is analytic for $\mathrm{Re}(z)<1$.
\end{enumerate}
\end{assumption}

The last two assumptions guarantee $A$-stability of the method, in other words, the method is stable on the entire negative half-complex plane $\mathbb{C}^-$.
Important examples of RK methods satisfying these assumptions are Radau IIA (with order $p = 2r - 1$ and stage order $q = r$) and Lobatto~IIIC methods (with order $p = 2r - 2$ and stage order $q = r-1$). 

Building on these Runge–Kutta methods with Assumption~\ref{assumption}, the associated convolution quadrature is constructed as follows. We define time grids $\tau:=\{t_k+c_ih\}_{k=0, \ldots,N}^{i=1, \ldots, r}$ and $\mathbf{t}_k:=\{t_k+c_ih\}_{i=1, \ldots,r}$ with $t_k=kh$, $t_k^i=t_k+c_ih$ and let $f:[0,T]\Flder\R^d$. Consider the collection of vectors $f(\mathbf{t}_k)=\mathbf{f}_k=\{f_k^i=f(t_k^i)\}_{i=1,\ldots,r}=\left(f(t_k^1),\ldots,f(t_k^r)\right)^\top\in\mathbb{R}^{r\times d}$ for $k=0,\ldots,N$. 
\textcolor{red}{
For $f:[0,T]\Flder\R^d$, the stage vector is $f_k=(f_k^1,\ldots,f_k^r)$ and each stage value $f_k^i\in \R^d$. Thus $\mathbf{f}_k\in (\R^d)^r=\R^{r\times d}$.
The convolution weights act only on the stage index and the spatial dimension is handled componentwise.}

Then, using the \textit{retarded} RKCQ methods, the approximation of the fractional integral $J_-^\alpha f(\mathbf{t}_k)$   is given by
\begin{equation}\label{RKConQua}
J_{-}^{\alpha}f(\mathbf t_k)\approx\mathcal{J}_{-}^{\alpha}f(\mathbf t_k):=\sum_{n=0}^{k}W_{k-n}^{(\alpha)}{\bf f}_{n}=\sum_{n=0}^{k}W_{k-n}^{(\alpha)}f({\bf t}_n)
\end{equation}
or explicitly
$$\mathcal{J}_{-}^{\alpha}{\bf f}_k= \sum_{n=0}^kW_{k-n}^{(\alpha)}\begin{pmatrix}
{f}^{1}_{n}\\
\vdots\\
{f}^{ r }_{n}
\end{pmatrix}.$$
Here, $K^{(\alpha)}$ is again as defined in~\eqref{LaplaceTrans} and the matrices $W_n^{(\alpha)}\in\R^{ r \times  r}$ are the  {\it convolution weights},  determined by the series expansion
 \begin{equation}\label{eq:series}
    K^{(\alpha)}\lp\frac{\gamma(z)}{h}\rp=\sum_{n=0}^{\infty}W_n^{(\alpha)}z^n,
 \end{equation}
where the matrix-valued function $\gamma(z)$, based on the underlying RK method, is given by 
\textcolor{red}{\[
\gamma(z)=\lp  A +\frac{z}{1-z}\mathds{1}\,\mathbf b^\top \rp^{-1}= A ^{-1}-z \frac{A ^{-1}\mathds{1}\mathbf  b^\top  A ^{-1}}{1-z\left(1-b^\top A^{-1}\mathds{1}\right)}.
\]
The identity follows from the Sherman–Morrison formula for the inverse of a rank-one perturbation of an invertible matrix~\cite{LehSaBook,LuOs}. For stiffly accurate implicit RK methods~\cite[Proposition 3.8]{HaWa}, such as Lobatto~IIIC, one has $b^\top A^{-1}\mathds{1}=1$. This is why
\[\gamma(z)=\lp  A +\frac{z}{1-z}\mathds{1}\,\mathbf b^\top \rp^{-1}= A ^{-1}-z A ^{-1}\mathds{1}\mathbf  b^\top  A ^{-1}.
\]
}
Using a different notation, which will be convenient later, the approximation at the stage time $t_k^i$ is defined by the $i$-th row of~\eqref{RKConQua} and can be computed   as \footnote{Again, the convolution quadrature can equivalently  be defined as
\[
\mathcal{J}_{-}^{\alpha}\mathbf f_k=\sum_{n=0}^{k}W_{n}^{(\alpha)}f(\mathbf t_{k-n})
\]
}
\begin{equation}\label{explicit_RKCQ}
    \mathcal{J}_{-}^{\alpha}f_k^i=\sum_{n=0}^{k} \sum_{j=1}^{r} \left[W_n^{(\alpha)}\right]_{j}^if_{k-n}^j, \quad i=1,\ldots, r.
\end{equation}
where $\left[W_n^{(\alpha)}\right]_{j}^i$ denotes the 
$(i,j)$-th entry of the weight matrix $W_n^{(\alpha)}$.
An efficient way to  compute the weights can be achieved by using fast Fourier transform (FFT) as described in~\cite{LuOs}. Following~\cite[Section 5.4]{LehSaBook}, the convolution weights $W_n^{(\alpha)}$ can be represented as a Cauchy integral formula as follows

\begin{align}
    W_n^{(\alpha)} &=\frac{1}{2\pi i}\int_{|z|=\lambda}K^{(\alpha)}\left(\frac{\gamma(z)}{h}\right)z ^{-n-1} \dd z\nonumber\\
 &=\lambda^{-n}\int_{0}^1 K^{(\alpha)}\left(\frac{\gamma(\lambda e^{-2\pi i \theta})}{h}\right)e ^{2\pi i \theta n} \dd \theta,\quad 0<\lambda<1.\nonumber\\
 &\approx \frac{\lambda^{-n}}{N+1}\sum_{\ell=0}^N K^{(\alpha)}\left(\frac{\gamma(z_\ell)}{h}\right)z_\ell^{-n}
\quad \text{(trapezoidal rule)}\label{TR}
\end{align}
where $z_\ell=\lambda e^{-2\pi i\ell/(N+1)}$. The  following  algorithm  summarizes  the  basic  steps  for  computing the convolution weights
\begin{algorithm}{}
\begin{algorithmic}[1]
\State {\bf Initial data:} $N$, $h$, $\gamma(z)$, $K^{(\alpha)}(s)$ and quadrature points $z_\ell,\, \ell =0,\ldots,N$
 \For {$\ell= 0,\ldots, N$} 
\State {\bf Diagonalize} $\gamma(z_\ell)=V_\ell D_\ell V_\ell^{-1}$
\State {\bf Apply} $K^{(\alpha)}$: $K^{(\alpha)}\left(\frac{\gamma(z_\ell)}{h}\right)=V_\ell K^{(\alpha)}(D_\ell/h) V_\ell^{-1}$,
\State {\bf Use} the FFT in equation~\eqref{TR}
\EndFor
\State {\bf Output:} $W_0^{(\alpha)},W_1^{(\alpha)},\ldots,  W_N^{(\alpha)}.$
\end{algorithmic}
\caption{Basic algorithm to compute the convolution weights}\label{alg:Fractional_weights_Algorithm}
\end{algorithm}

A good choice of $\lambda$ is $\lambda=\varepsilon^{1/(2N+1)}$ yields an error in $W_n^{(\alpha)}$ of size $\sqrt{\varepsilon}$; see~\cite{Lubich2}. Another representation of the weights $W_n^{(\alpha)}$  can be found in~\cite{BaLo}, leading also to an efficient algorithm approximating fractional integrals.

For stiffly accurate Runge–Kutta methods such as Lobatto~IIIC and Radau IIA, one has 
$$b^\top A^{-1}=e_r^\top,\quad c_r=1,\qquad \text{where}\quad  e_r^\top=(0,\ldots,0,1).$$
Therefore, the last row of~\eqref{explicit_RKCQ} yields the approximation to $\mathcal{J}_{-}^{\alpha}f_k^r=\mathcal{J}_{-}^{\alpha}f(t_{k+1})$. Thus, the numerical solution at the grid point $t_{k+1}$ is given by
 \begin{equation}\label{explicit_RKCQ2}
\mathcal{J}_{-}^{\alpha}f(t_{k+1})=\mathcal{J}_{-}^{\alpha}f_k^r=b^\top A^{-1}\sum_{n=0}^{k}W_{k-n}^{(\alpha)}f({\bf t}_n).
 \end{equation}
Now, we state the  convergence order of the RKCQ~\eqref{explicit_RKCQ2} as shown in~\cite[Theorem 3]{BaLuMar}. Here, we omit the regularity conditions on the general Laplace transform $K(s)$ for which the original theorem is established, since our focus is on $K^{(\alpha)}(s)$ as defined in~\eqref{LaplaceTrans}, which already satisfies these conditions.  
\begin{theorem}[\cite{BaLuMar}]\label{thm:RKCQ_convergence}
Let $K^{(\alpha)}$ be the Laplace transform  of the kernel which  is analytic in the half-plane $\mathrm{Re} (s) > \sigma > 0$, such that for
some real exponent $\mu$ and bounding factor $C_K(\sigma)> 0$, the operator norm is bounded as follows:
\begin{equation}\label{eq:hyperbolic_symbol}
    \left\|K^{(\alpha)}(s)\right\|\leq C_K(\sigma) |s|^\mu\quad \text{for all}\quad \mathrm{Re} (s)>\sigma.
\end{equation}
Let the underlying Runge-Kutta method satisfy Assumption~\ref{assumption}. Assume that   $\ell> \max(p+\mu+1,p,q+1)$ and let $f\in C^\ell([0,T],\R)$ satisfy $f(0)=f'(0)=\cdots=f^{(\ell-1)}(0)=0$. Then there exist constants $\bar h>0$ and $C>0$ such that, for all $h\le \bar h$ and $t\in[0,T]$, the RKCQ method~\eqref{explicit_RKCQ2}  satisfies
\begin{equation}\label{RK_error_bound}
\left\|J^{\alpha}_{-}f(t_{k+1})-\mathcal{J}_{-}^{\alpha}f(t_{k+1})\right\|\leq C h^{\min(p,q+1-\mu)}\left(\left|f^{(\ell)}(0)\right|+\int_0^t\left|f^{(\ell+1)}(s)\right|\dd s\right).
\end{equation}
The constant $C$ depends on $\bar h,T$ and the  chosen  RK method, but is independent of $h$ and $f$.
\end{theorem}

The previous theorem is applicable  to the fractional integral~\eqref{RLInt:a} when $\mu=\alpha<0$ and to the fractional derivative~\eqref{RLDer:a} in the case $\mu>0$. For the fractional derivative $D_{-}^{\alpha}$ (i.e., $\mu=\alpha$), the convolution quadrature is defined by
\begin{equation}\label{RKConQua2}
D_{-}^{\alpha}f(\mathbf t_k)\approx\mathcal{D}_{-}^{\alpha}\mathbf f_k:=\sum_{n=0}^{k}W_{k-n}^{(-\alpha)}{\bf f}_{n}.
\end{equation}

\begin{remark}\label{rem:full-order}
In the sectorial case~\cite{LuOs}, the stage rate is $q+1$, independent of the decay exponent $\mu>0$ of the kernel $K$. The rate for the last stage is determined by both error terms $h^p$ (classical order) and $h^{q+1+\mu} |\log (h)|$ (stage order improved by symbol decay). Theorem~\ref{thm:RKCQ_convergence} does not cover an internal stage error estimate for the kernel~\eqref{eq:hyperbolic_symbol}.  The estimate in~\eqref{RK_error_bound} attains the classical convergence order  $p$ if and only if
\begin{equation} \label{eq:full-order-cond}
\mu \leq q + 1 - p.
\end{equation}
Since, for all standard methods, one has $p \geq q + 1$, this condition  requires $\mu \leq 0$. Thus, for kernels of the form $K(s)=s^\mu$ with $\mu>0$, the  convergence rate $q+1-\mu$ remains strictly below $p$, and the full classical order cannot be achieved. Moreover, for methods with $p > q + 1$, such as the $r$-stage Lobatto~IIIC with $r \geq 3$, for which $p = 2r-2$ and $q = r-1$, even the case $\mu = 0$ gives stage rate $q+1 = r < 2r-2 = p$. In Figure~\ref{fig:expected_convergence_RKCQ_045_075}, our numerical experiments  show that all stages converge at the same rate $q+1-\mu$. 
\end{remark}

We apply Algorithm~\ref{alg:Fractional_weights_Algorithm} to the kernel $K(s)=s^\alpha$  to compute the weights of the $r$-stage Lobatto~IIIC (see Table~\ref{tab:Lobatto}). These weights are then used to evaluate the fractional derivatives of $f(t)= e^{-t}  \sin^6(t)$. 
The expected convergence order is $\mathcal O(h^{r-\alpha})$, and this rate is confirmed by the numerical experiments.
\begin{figure}[htbp]
\centering
\includegraphics[width=1\textwidth]{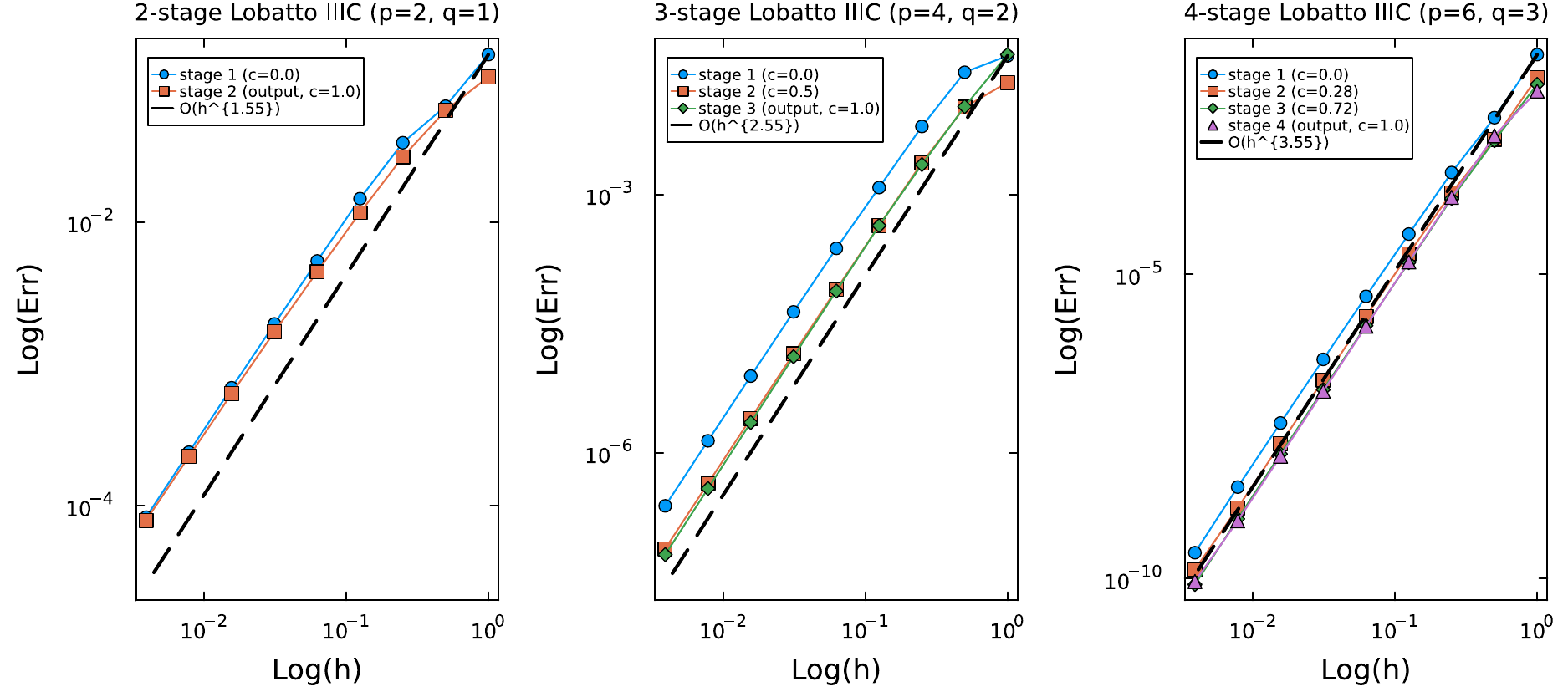}\\
\includegraphics[width=1\textwidth]{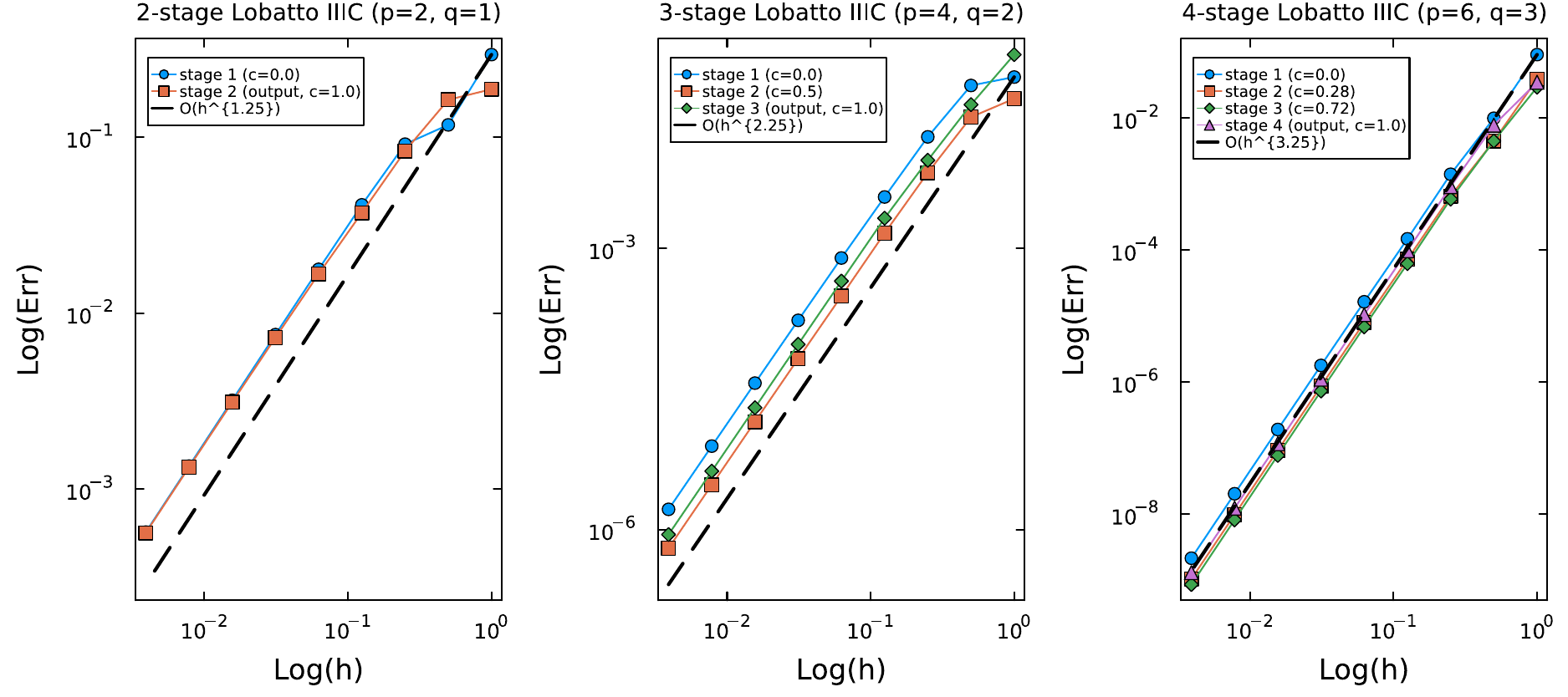}
\caption{Convergence rates of CQ based on Lobatto~IIIC for
$K(s)=s^\alpha$, applied to the approximation of fractional derivatives  of $ e^{-t}  \sin^6(t)$ on $t\in[0,4]$, $\alpha=0.45$ (top) and  $\alpha=0.75$ (bottom). The plot confirms that for $r=2,3,4$, the error behaves as $\mathcal O(h^{r-\alpha})$.}
\label{fig:expected_convergence_RKCQ_045_075}
\end{figure}


\section{Fractional Variational Integrators with RKCQ}\label{sec:FVI}
In this section we derive fractional integrators for the restricted variational principle using  Runge-Kutta convolution quadratures. We first recall the restricted variational principle as presented in~\cite{JiOb1,JiOb2}, which forms the foundation for the construction of such integrators.

\subsection{Continuous Setting of Restricted Variational Principle}\label{sec:restrectedCOV}
Let  $x,y\in AC^2([0,T],\mathbb R^d)$ and  $L:\mathsf{T}\R^d\Flder \R$ be a $C^2$-Lagrangian (using the natural identification $\mathsf{T}\R^d\cong\R^d\times\R^d$). Define the augmented fractional Lagrangian
\begin{equation}\label{GenerLag}
\begin{array}{rcl} L^{(\alpha)}:\R^d\times\R^d\times\R^d\times\R^d\times\R^d\times\R^d &  \Flder  & \R \\
 (x,\,y,\,\dot x,\,\dot y,\, D^{\alpha}_{-}x,\, D^{\alpha}_{+}y)    & \mapsto&L^{(\alpha)}(x,y,\dot x,\dot y, D^{\alpha}_{-}x,D^{\alpha}_{+}y)\\ [2ex]
  L^{(\alpha)}(x,\,y,\,\dot x,\,\dot y,\, D^{\alpha}_{-}x,\, D^{\alpha}_{+}y)=L(x,\dot x)+&&\hspace{-1cm}L(y,\dot y)-\rho\,D^{\alpha}_{-}x\,D^{\alpha}_{+}y,\quad \rho >0.
\end{array}
\end{equation}
Accordingly, we define  the relevant action $\mathcal{L}: AC^2([0,T])\times AC^2([0,T])\to \R$ by  
$\mathcal{L}(x,y)=\Lc(x,y)+\Lf(x,y)$ where
\begin{equation}\label{FracAction}
\begin{split}
\Lc(x,y)&=\int_0^T(L(x(t),\dot x(t))+L(y(t),\dot y(t))\,)\, \dd t\\
\Lf(x,y)&=-\rho\,\int_0^T D^{\alpha}_{-}x(t)\,D^{\alpha}_{+}y(t)\,\dd t.
\end{split}
\end{equation}
Moreover,  we consider a family of varied curves $x_\epsilon(t)$ and $y_\epsilon(t)$
\begin{equation}\label{VariedCurves}
x_{\epsilon}(t)=x(t)+\epsilon\, \delta x(t),\,\,\,y_{\epsilon}(t)=y(t)+\epsilon\, \delta y(t),\quad \epsilon>0
\end{equation}
subject to the restricted variations introduced in~\cite{JiOb2}, namely
\begin{equation}\label{VariedCurves2}
\delta x(t)=\delta y(t)\quad \text{for all } t\in[0,T],
\end{equation}
with  $\delta x(0)=\delta x(T)=0$, meaning that the two trajectories $x$ and $y$ are varied by the same perturbation. Throughout, we assume the endpoint $x(0),\,x(T),\,y(0),\,y(T)$ are fixed. Then  the restricted fractional Euler-Lagrange equations hold.
\begin{theorem}[FEL equations]\label{ContTheo}
The equations
\begin{subequations}\label{ContFracDamp}
\begin{align}
\frac{\dd}{\dd t}\lp\frac{\der L(x,\dot x)}{\der\dot x}\rp-\frac{\der L(x,\dot x)}{\der x}=-\rho\,D^{2\alpha}_{-}x,\label{ContFracDamp:a}\\
\frac{\dd}{\dd t}\lp\frac{\der L(y,\dot y)}{\der\dot y}\rp-\frac{\der L(y,\dot y)}{\der y}=-\rho\,D^{2\alpha}_{+}y,\label{ContFracDamp:b}
\end{align}
\end{subequations}
are  {\rm sufficient} conditions for the extremals of $\L(x,y)$~\eqref{FracAction} under the  restricted variations~\eqref{VariedCurves2}.
\end{theorem}
\textcolor{blue}{\begin{remark}
In the doubled-variable formulation~\cite{JiOb2}, the curves $x(t)$ and $y(t)$ represent a forward trajectory and its time-reversed counterpart, thereby allowing fractional dissipative systems to be incorporated into a variational framework. The unrestricted variational principle produces necessary and sufficient stationarity conditions, but the resulting coupled equations do not represent the intended physical model.
Imposing  $\delta x=\delta y$ restricts perturbations to those respecting the time-reversal structure. The variational principle then provides only sufficient conditions, yielding the physically relevant fractional damping equations. This loss of necessity is a deliberate choice to obtain correct dissipative dynamics.
\end{remark}}

\subsection{Discrete Setting Based on RKCQ}\label{subsec:Discrete_Setting}
We analogously define the \textit{advanced} RKCQ by 
\begin{equation}\label{RKConQuaPlus}
J_{+}^{\alpha}f(\mathbf t_k)\approx\mathcal{J}_{+}^{\alpha}f(\mathbf t_k):=\sum_{n=0}^{N-k}\left[W_{n}^{(\alpha)}\right]^\top\mathbf{f}_{k+n}.
\end{equation}
 where the transpose relationship  in the equality can be followed from the definition of the series,
\[
\left[K^{(\alpha)}\lp\frac{\gamma(z)}{h}\rp\right]^\top=\left[\sum_{n=0}^{\infty} W_n^{(\alpha)}z^n\right]^\top = \sum_{n=0}^{\infty} \left[W_n^{(\alpha)}\right]^\top z^n.
\]

\begin{remark}
In the formulation of the discrete restricted Hamilton principle, all variables must be treated consistently within the discrete framework. Therefore, the function 
$f$ should be represented by its discrete approximation 
$f_\dis=\{\mathbf f_k\}_{0:N}$, defined at the discrete time. Accordingly, the convolution quadratures~\eqref{RKConQua} and~\eqref{RKConQuaPlus}  should be constructed using these discrete values, rather than using the evaluation of the continuous function $f(t)$ at grid time $t_k^i$. This ensures consistency with the fully discrete variational framework and avoids mixing continuous and discrete representations.
\end{remark}
\textcolor{red}{With these ingredients, together with Assumption~\eqref{assumption}, we can establish the asymmetric integration by parts for the RKCQ which we need to derive the restricted discrete dynamics.}
\textcolor{red}{For simplicity, we state the following lemma in the scalar case 
$d=1$, so that  $f:[0,T]\to \R$.}


\begin{lemma}[Semigroup property for RKCQ]\label{lemma:Semi_group}
Consider a discrete series $\left\{ {\bf f}_k=(f_k^1,\ldots,f_k^r)^\top\right\}_{0:N}$. Then for $\alpha>0$, the RKCQ operators satisfy\footnote{\color{blue}In our applications, each stage value is vector-valued 
$f_k^i\in \R^d$, hence ${\bf f}_k\in \big(\R^d\big)^r\cong \R^{r\times d}$. All identities hold componentwise (apply the lemma to each component of $\R^d$).}
\[
\mathcal J_-^\alpha \circ \mathcal J_-^\alpha  {\bf f}_k
=
\mathcal J_-^{2\alpha} {\bf f}_k,\qquad \sigma\in \{-,+\}.
\]
\end{lemma}
\begin{proof}
For simplicity, we present the proof in the scalar case; the vector-valued case follows componentwise. Let  $\sigma=-$, the proof for $\sigma=+$ is similar. By definition,
\[
\mathcal J_-^\alpha {\bf f}_k
=
\sum_{n=0}^{k} W_n^{(\alpha)} {\bf f}_{k-n},
\]
where the matrix-valued weights $W_n^{(\alpha)} \in \mathbb{R}^{r\times r}$ are obtained by the generating series
\[
\sum_{n=0}^{\infty}
W_n^{(\alpha)} z^n
=
K^{(\alpha)}\!\left(\frac{\gamma(z)}{h}\right),
\qquad
K^{(\alpha)}(s)=s^{-\alpha}.
\]
Firstly, with $K^{(\alpha)}(s)=s^{-\alpha}$, we have
$
K^{(\alpha)}(A)K^{(\alpha)}(A)
=
K^{(2\alpha)}(A)$ for 
$A=\gamma(z)/h$. Thus
\[
\left(
\sum_{n=0}^{\infty}
W_n^{(\alpha)} z^n
\right)\left(
\sum_{m=0}^{\infty}
W_m^{(\alpha)} z^m
\right)
=
K^{(2\alpha)}\!\left(\frac{\gamma(z)}{h}\right):=\sum_{\ell=0}^{\infty}
W_\ell^{(2\alpha)} z^\ell.
\]
By the Cauchy product formula for matrix-valued power series, and uniqueness of power series coefficients, we obtain
\begin{equation}\label{proof:semigroup}
     W_\ell^{(2\alpha)}
=\sum_{n =0}^{\ell}W_{n}^{(\alpha)} W_{\ell-n }^{(\alpha)},\qquad \ell \geq 0.
\end{equation}
which is valid under Assumption~\ref{assumption}. Secondly, applying $\mathcal J_-^\alpha$ to $\mathcal J_-^\alpha {\bf f}_k$ yields
\[
(\mathcal J_-^\alpha\circ \mathcal J_-^\alpha {\bf f})_k
=
\sum_{n=0}^{k}
W_n^{(\alpha)}
(\mathcal J_-^\alpha {\bf f})_{k-n}.
\]
Substituting the definition of $J_-^\alpha {\bf f}_k$ and rearranging the sums gives
\begin{align*}
(\mathcal J_-^\alpha \circ\mathcal J_-^\alpha {\bf f})_k
=
\sum_{n=0}^{k}
W_n^{(\alpha)}
\left(
\sum_{m =0}^{k-n}
W_m ^{(\alpha)} 
{\bf f}_{k-n-m}
\right)  
=^{1}\sum_{n=0}^{k}
\sum_{\ell=n}^{k} W_{n}^{(\alpha)} W_{\ell-n}^{(\alpha)} {\bf f}_{k-\ell}.
\end{align*}
where in $=^1$ we used $\ell=n+m$. On the other hand, for a fixed $\ell$ (which can range from $0$ to  $k$) the constraint  $n\leq \ell\leq k$ with  $0\leq n\leq \ell$ shows that the last quantity is equivalent to
\begin{align*}
(\mathcal J_-^\alpha \circ\mathcal J_-^\alpha {\bf f})_k=
\sum_{\ell=0}^{k} \sum_{n=0}^{\ell}W_{n}^{(\alpha)} W_{\ell-n}^{(\alpha)} {\bf f}_{k-\ell}.
\end{align*}
With identity~\eqref{proof:semigroup}, the above shows the semigroup property, which completes the proof.
\end{proof}

\begin{lemma}[Asymmetric integration by parts]\label{ConvPropertiesRK1}
Consider two discrete series $\left\{ {\bf f}_k\right\}_{0:N}, \left\{ {\bf g}_k\right\}_{0:N}$, where ${\bf f}_k=(f_k^1,\ldots,f_k^r)^\top\in \R^r$ and ${\bf g}_k=(g_k^1,\ldots,g_k^r)^\top\in \R^r$. Then the following property holds:
 \begin{equation}\label{AsymmetricInt:2}
\sum_{k=0}^{N}\Big\langle {\bf g}_k, \mathcal{J}_{-}^{\alpha}{\bf f }_k\Big\rangle=\sum_{k=0}^{N}\Big\langle \mathcal{J}_{+}^{\alpha}{\bf g}_k, {\bf f }_k\Big\rangle.
\end{equation}  
\end{lemma}
\begin{proof}
Using the definition~\eqref{RKConQua} and employing lower index for the row vectors gives
\[
\sum_{k=0}^{N}\Big\langle {\bf g}_k, \mathcal{J}_{-}^{\alpha}{\bf f }_k\Big\rangle=
\sum_{k=0}^{N}\Big\langle{\bf g}_k,\sum_{n=0}^{k}W_n^{(\alpha)}\,{\bf f}_{k-n}\Big\rangle=\sum_{k=0}^{N}\sum_{n=0}^{k}\Big\langle{\bf g}_k,W_n^{(\alpha)}\,{\bf f}_{k-n}\Big\rangle
\]
Swap the order of summation. Since the index set is
$\{(k,n):0\le k\le N,\ 0\le n\le k\}$, for a fixed $n\in\{0,\dots,N\}$
the index $k$ runs from $k=n$ to $k=N$. Hence
\begin{align*}
\sum_{k=0}^{N}\Big\langle {\bf g}_k, \mathcal{J}_{-}^{\alpha}{\bf f }_k\Big\rangle&=\sum_{n=0}^{N}\sum_{k=n}^{N}\Big\langle{\bf g}_k,W_n^{(\alpha)}\,{\bf f}_{k-n}\Big\rangle\\
&=^1\sum_{n=0}^{N}\sum_{m=0}^{N-n}\Big\langle{\bf g}_{m+n},W_n^{(\alpha)}\,{\bf f}_{m}\Big\rangle\\
&=^2\sum_{n=0}^{N}\sum_{m=0}^{N-n}\Big\langle {W_n^{(\alpha)}}^\top{\bf g}_{m+n},\,{\bf f}_{m}\Big\rangle
\end{align*}
where setting  $m=k-n$ in $=^1$  and  the identity for the inner product $\big\langle x,Ay\big\rangle=\big\langle A^\top x,y\big\rangle$ in $=^2$. The index constraint $0\le m\le N-n$ is equivalent to $0\le n\le N-m$, hence the summation over
$\{(n,m):0\le n\le N,\ 0\le m\le N-n\}$ coincides with the summation over
$\{(m,n):0\le m\le N,\ 0\le n\le N-m\}$. Therefore, swapping the sums again yields
\begin{align*}
\sum_{n=0}^{N}\sum_{m=0}^{N-n}\Big\langle {W_n^{(\alpha)}}^\top{\bf g}_{m+n},\,{\bf f}_{m}\Big\rangle
&=\sum_{m=0}^{N}\sum_{n=0}^{N-m}\Big\langle {W_n^{(\alpha)}}^\top{\bf g}_{m+n},\,{\bf f}_{m}\Big\rangle\\
&=\sum_{m=0}^{N}\Big\langle \sum_{n=0}^{N-m}{W_n^{(\alpha)}}^\top{\bf g}_{m+n},\,{\bf f}_{m}\Big\rangle\\
&=\sum_{m=0}^{N}\big\langle \mathcal J_+^\alpha{\bf g}_m,\,{\bf f}_m\big\rangle.
\end{align*}
This completes the proof.
\end{proof}

Before deriving fractional variational integrators with RKCQ \S\ref{sec:FVI}, which are fully characterized by Butcher Tableau, we first establish the following technical lemma.  
\begin{lemma}[Stage-weighted asymmetric integration by parts]\label{ConvPropertiesRK2}
Consider two discrete series $\left\{ {\bf f}_k\right\}_{ 0:N}, \left\{ {\bf g}_k\right\}_{0:N}$ with ${\bf f}_k,{\bf g}_k\in \R^{r\times d}$. Let $B=\diag(b_1,\ldots,b_r)$ and let $\mathcal{J}_{-}^{\alpha},\mathcal{J}_{+}^{\alpha}$ be the retarded/advanced RKCQ fractional integrals. Then
\begin{equation}\label{AsymmetricInt:3}
\sum_{k=0}^{N}\Big\langle {\bf g}_k, \mathcal{J}_{-}^{\alpha}(B{\bf f })_k\Big\rangle=\sum_{k=0}^{N}\Big\langle B\mathcal{J}_{+}^{\alpha}{\bf g}_k, {\bf f }_k\Big\rangle.
\end{equation}
Here $B$ acts componentwise on the stages, i.e. $(B {\bf f }_k)^i=b_i {\bf f }_k^i$ for $i=1,\dots,r$.
\end{lemma}

\textcolor{blue}{
\begin{proof}
Applying Lemma~\ref{ConvPropertiesRK1} componentwise to the pair $(B{\bf f}_k,\,{\bf g}_k)$:
\[
\sum_{k=0}^N \Big\langle {\bf g}_k,\,(\mathcal J_-^\alpha (B {\bf f}))_k\Big\rangle
=
\sum_{k=0}^N \Big\langle (\mathcal J_+^\alpha {\bf g})_k,\,(B {\bf f})_k\Big\rangle .
\]
Since $B^\top=B$ (diagonal), we have $\langle u,\,Bv\rangle=\langle Bu,\,v\rangle$. Hence
\[
\sum_{k=0}^N \Big\langle (\mathcal J_+^\alpha {\bf g})_k,\,(B {\bf f})_k \Big\rangle
=
\sum_{k=0}^N  \Big\langle (B \mathcal J_+^\alpha {\bf f})_k,\, {\bf f}_k \Big\rangle,
\]
which is~\eqref{AsymmetricInt:3}.
\end{proof}
}
 In what follows, we set $r=s+1$. Moreover, we focus on Lobatto~IIIC methods for which $c_1=0$ and $c_{r}=1$. Table~\ref{tab:Lobatto} lists the Lobatto~IIIC Butcher tableaux for $r=2,3,4$. In addition, we denote by $W_k^{(\alpha)}$ the RKCQ weights associated with the fractional derivatives  (we use this notation to avoid repeatedly writing weights with a negative order).
\begin{table}[H]
\centering
$\setlength{\tabcolsep}{20pt}
\renewcommand{\arraystretch}{1.5}\begin{array}{c|cc}
0  &  \frac{1}{2} & - \frac{1}{2} \\
   1  & \frac{1}{2} &\phantom{-}\frac{1}{2} \\ \hline 
   & \frac{1}{2}& \frac{1}{2}
\end{array}\hspace{2cm}\begin{array}{c|ccc}
0  &  \frac{1}{6} & - \frac{1}{3} & \phantom{-}\frac{1}{6} \\
   \frac{1}{2}  &\frac{1}{6} &  \phantom{-}\frac{5}{12} & -\frac{1}{12} \\
 1   &\frac{1}{6} & \phantom{-} \frac{2}{3} & \phantom{-}\frac{1}{6}  \\ \hline 
 &\frac{1}{6} &  \phantom{-}\frac{2}{3} & \phantom{-}\frac{1}{6} 
\end{array}\hspace{2cm}\begin{array}{c|cccc}
0  &  \frac{1}{12} & - \frac{\sqrt{5}}{12} & \frac{\sqrt{5}}{12}&-\frac{1}{12} \\
\frac{5-\sqrt{5}}{10} &  \frac{1}{12} &  \frac{1}{4} & \frac{10-7\sqrt{5}}{60}&\phantom{-}\frac{\sqrt{5}}{60} \\
\frac{5+\sqrt{5}}{10}  &  \frac{1}{12} & \frac{10+7\sqrt{5}}{60} & \frac{1}{4}&-\frac{\sqrt{5}}{60} \\
1  &  \frac{1}{12} &  \frac{5}{12} & \frac{5}{12}&\phantom{-}\frac{1}{12} \\ \hline
 &\frac{1}{12} &  \frac{5}{12} & \frac{5}{12} & \phantom{-}\frac{1}{12}
\end{array}$
\caption{Lobatto~IIIC methods with $r = 2$, $r = 3$  and $r = 4.$}
\label{tab:Lobatto}
\end{table}
\textcolor{blue}{
We consider two discrete series $x_\dis^{}:=\{\mathbf{x}_k\}_{0:N}$ and  $y_\dis^{}:=\{\mathbf{y}_k\}_{0:N}$
where each $\mathbf{x}_k$ and $\mathbf{y}_k$ is a vector of stage values. Concretely,
\begin{itemize}
    \item $\mathbf{x}_k:=(x_k^1,\ldots,x_k^{s+1})^\top$ and $\mathbf{y}_k:=(y_k^1,\ldots,y_k^{s+1})^\top$;
    \item each stage component $x_k^i$ (and $y_k^i$) represents the approximation of the trajectory at the Runge–Kutta node $t_k+c_ih$ inside the time interval 
$[t_k,t_{k+1}]$.
\end{itemize}}
Because Lobatto~IIIC satisfies 
$c_1=0$ and $c_r=1$, the main stages correspond to the endpoints $t_k$ and $t_{k+1}$. Therefore, continuity at these nodes is enforced by imposing $x_{k-1}^{s+1}=x_{k}^{1}$ and similarly $y_{k-1}^{s+1}=y_{k}^{1}$. Regarding the approximation of~\eqref{FracAction}, we consider the following discrete action:  
\begin{equation}\label{DiscFracActionHORK}
\begin{split}
\mathcal{L}_\dis(x_\dis^{},y_\dis^{})=\Lc_\dis&(x_\dis^{},y_\dis^{})+\Lf_\dis(x_\dis^{},y_\dis^{}), \\
\Lc_\dis(x_\dis^{},y_\dis^{})=\sum_{k=0}^{N-1}\big[L_\dis({\bf x}_k)+L_\dis({\bf y}_k)\big],&\quad \Lf_\dis(x_\dis^{},y_\dis^{})=-\rho\,h\,\,\sum_{k=0}^{N}\left\langle\widetilde{\mathcal{D}}^{\alpha}_{+}{\bf y}_k,\mathcal{D}^{\alpha}_{-}{\bf x}_k\right\rangle,\\
L_\dis({\bf x}_k)=h\sum_{i=1}^rb_iL(x_\dis^{}(c_i\,h;k),\,\dot x_\dis^{}(c_i\,h;k)),&\,\,\,\,\,\, L_\dis({\bf y}_k)=h\sum_{i=1}^rb_iL(y_\dis^{}(c_i\,h;k),\dot y_\dis^{}(c_i\,h;k)),
\end{split}
\end{equation}
where the conservative contribution is constructed using the standard Galerkin with Lobatto quadrature (see \S\ref{HO-Action}). For the fractional contribution, we employ the RK convolution quadratures~\eqref{RKConQua},~\eqref{RKConQuaPlus} with $\widetilde{\mathcal{D}}^{\alpha}_{+}{\bf y}_k:=\mathcal{D}^{\alpha}_{+}B{\bf  y}_k$. 

Finally, to formulate the restricted discrete variational principle, we introduce discrete varied curves that mirror the continuous restriction ``both trajectories have the same variation''. Specifically, for $\epsilon>0$, we perturb both discrete curves by the same discrete variation at every step and stage, namely $\delta \mathbf{y}_k=\delta \mathbf{x}_k:=(\delta x_k^1,\ldots,\delta x_k^{s+1})^\top$. In other words, the full varied discrete series is given by
\begin{equation}\label{RestrictedVariedInnerRK}
x_\dis^{\epsilon}=\left\{\mathbf{x}_k+\epsilon \delta \mathbf{x}_k\right\}_{0:N},\qquad y_\dis^{\epsilon}=\left\{\mathbf{y}_k+\epsilon \delta \mathbf{x}_k\right\}_{0:N}.
\end{equation}
Naturally, we  set 
$x_N^{i}=0$ for $i=2,\ldots,s+1$ and  $\delta\mathbf  x_N=0$ (equivalent for $y$), besides $\delta x_{0}^{1}=\delta y_{0}^{1}=\delta x_{N-1}^{s+1}(=\delta x_{N}^{1})=0=\delta y_{N-1}^{s+1}(=\delta y_{N}^{1})$, since $x_{0}^{1},\,y_{0}^{1},\,x_{N-1}^{s+1}=x_{N}^{1},\,y_{N-1}^{s+1}=y_{N}^{1}$  are fixed.
The following lemma is needed for the proof of the main theorem. In its proof and in the rest of the paper, we adopt the notation $[A]^i$ for the $i$-th row of a matrix $A$.
\begin{lemma}\label{VariationCommRK}
According to the definitions~\eqref{RKConQua},~\eqref{RKConQuaPlus} and considering varied curves~\eqref{RestrictedVariedInnerRK}, we have\footnote{\textcolor{blue}{The symbol $\delta$ denotes the variational  operator with respect to the  parameter 
$\epsilon$, evaluated at $\epsilon=0$, associated with the $\epsilon$-dependent curve. More precisely, for any  $\epsilon$-dependent quantity $F(x^\epsilon)$, we define $\delta F = \frac{d}{d \epsilon}\big|_{\epsilon=0}F(x^\epsilon)$.
}}
\[
\delta \mathcal{D}^{\alpha}_{-}{\bf x}_k = \mathcal{D}^{\alpha}_{-}{\bf \delta x}_k\qquad\text{and}\qquad \delta \mathcal{D}^{\alpha}_{+}{\bf y}_k = \mathcal{D}^{\alpha}_{+}{\bf \delta y}_k.
\]
\end{lemma}
\textcolor{blue}{\begin{proof}
We note that, in the RKCQ framework, the weights  $W_n^{(\alpha)}$ are determined entirely by the step size 
$h$, the underlying RK method $(A, \mathbf b, \mathbf c)$ through the generating function $\gamma(z)$ and the kernel $K^{(\alpha)}(s)$, in particular, they are independent of the discrete trajectory $\{\mathbf x_k\}_{0:N}$. So, if we pick $\delta \mathcal{D}^{\alpha}_{-}{\bf x}_k$, it follows that
\[
\delta \mathcal{D}^{\alpha}_{-}{\bf x}_k=\frac{\dd}{\dd\epsilon}\bigg|_{\epsilon=0}\mathcal{D}^{\alpha}_{-}{\bf x}_k^{\epsilon}
=\frac{\dd}{\dd\epsilon}\bigg|_{\epsilon=0} \sum_{n=0}^{k}W_{k-n}^{(\alpha)}(\mathbf x_{n}+\epsilon\,\delta \mathbf x_{n})=\sum_{n=0}^{k}W_{k-n}^{(\alpha)}\delta \mathbf x_{n}=\mathcal{D}^{\alpha}_{-}\,({\bf \delta x}_k).
\]
The proof for $\delta \mathcal{D}^{\alpha}_{+}{\bf y}_k$ is similar. 
\end{proof}}
\begin{theorem}[FDEL equations]\label{thm:maintheorem}
Let us consider two discrete series $\left\{{\bf x}_k\right\}_{0:N}$ and $\left\{{\bf y}_k\right\}_{0:N}$ such that $x_N^{2:s+1}=y_N^{2:s+1}=0$. Then, the equations
\begin{subequations}\label{FinalRK}
\begin{align}
&D_{s+1}L_\dis(x_{k-1}^{1},\ldots,x_{k-1}^{s+1})+ D_1L_\dis(x_k^1,\ldots,x_k^{s+1})\nonumber\\
&\hspace{3cm}-\rho h\left(b_1\left[\mathcal{D}^{2\alpha}_{-}{\bf x}_k\right]^1 +b_{s+1}\left[ \mathcal{D}^{2\alpha}_{-}{\bf x}_{k-1}\right]^{s+1}\right)=0,\quad  k=1,\ldots, N-1,\label{FinalRK1}\\
&D_iL_\dis(x_k^1,\ldots,x_k^{s+1})-\rho h b_i \left[\mathcal{D}^{2\alpha}_{-}{\bf x}_k\right]^i=0,\,\hspace{1cm}\,\,\quad k=0,\ldots, N-1,\,\quad i=2,\ldots,s,\label{FinalRKi}\\
&D_{s+1}L_\dis(y_{k-1}^{1},\ldots,y_{k-1}^{s+1})+ D_1L_\dis(y_k^0,\ldots,y_k^s)\nonumber\\
&\hspace{3cm}-\rho h\left(\left[\mathcal{D}^{2\alpha}_{+}B{\bf  y}_k\right]^1 +\left[\mathcal{D}^{2\alpha}_{+}B{\bf  y}_{k-1}\right]^{s+1}\right)=0,\quad  k=1,\ldots, N-1,\label{FinalRKy1}\\
&D_iL_\dis(y_k^{1},\ldots,y_k^{s+1})-\rho h \left[\mathcal{D}^{2\alpha}_{+}B{\bf y}_k\right]^i=0,\,\hspace{1cm}\,\,\quad k=0,\ldots, N-1,\,\quad i=2,\ldots,s,\label{FinalRKyi}
\end{align}
\end{subequations}
 are sufficient conditions for the extremals of~\eqref{DiscFracActionHORK} under restricted calculus of variations~\eqref{RestrictedVariedInnerRK}.
\end{theorem}

\begin{proof}
\textcolor{blue}{Let $\left\{{\bf x}_k\right\}_{0:N}$ and $\left\{{\bf y}_k\right\}_{0:N}$ be  two discrete series satisfying $x_N^{2:s+1}=y_N^{2:s+1}=0$. Under the restricted variations $\delta \mathbf x_k=  \delta\mathbf  y_k$ for all $k$,  the 
variation of  the conservative part is
\begin{equation}\label{eq:variation_proof}
\delta\Lc_\dis(x_\dis^{},y_\dis^{})=\sum_{k=0}^{N-1}\sum_{i=1}^{s+1}\lp D_iL_\dis(\mathbf x_k)+D_iL_\dis(\mathbf y_k)\rp \cdot \delta x_k^{i},
\end{equation}
where $D_iL_\dis(\mathbf x_k)\cdot\delta x_k^{i}:=\frac{\partial L_\dis(\mathbf x_k)}{\partial x_k^{i}}\cdot\delta x_k^{i}$. For the fractional part, we employ the Leibniz rule of the derivative and Lemma~\ref{VariationCommRK}, we obtain
\[
\delta\Lf_\dis(x_\dis^{},y_\dis^{})=-\rho\,h\sum_{k=0}^{N}\left\langle\widetilde{\mathcal{D}}^{\alpha}_{+}{\bf \delta\, y}_k,\mathcal{D}^{\alpha}_{-}{\bf x}_k\right\rangle -\rho\,h\sum_{k=0}^{N}\left\langle\widetilde{\mathcal{D}}^{\alpha}_{+}{\bf y}_k,\mathcal{D}^{\alpha}_{-}{\bf \delta\, x}_k\right\rangle.
\]
Applying Lemma~\ref{ConvPropertiesRK2} to each term  and taking into account that $\delta {\bf x}_k = \delta {\bf y}_k$ yields
\begin{equation}\label{sums}
    \delta\Lf(x_\dis^{},y_\dis^{})=-\rho\,h\sum_{k=0}^{N}\left\langle{\bf \delta\, x}_k,B\mathcal{D}^{2\alpha}_{-}{\bf x}_k\right\rangle -\rho\,h\sum_{k=0}^{N}\left\langle\mathcal{D}^{2\alpha}_{+}(B{\bf  y})_k,{\bf \delta\, x}_k\right\rangle
\end{equation}
where Lemma~\ref{lemma:Semi_group} (semigroup property) is used to identify $\mathcal D_-^{2\alpha}=\mathcal D_-^\alpha\circ \mathcal D_-^\alpha$ and $\mathcal D_+^{2\alpha}=\mathcal D_+^\alpha\circ \mathcal D_+^\alpha$. From~\eqref{RestrictedVariedInnerRK}, we have $\delta {\bf x}_N=0$, and the sums~\eqref{sums} run from $k=0$ to $N-1$.
Using the inner product, the terms in~\eqref{sums} can be written explicitly as
\begin{align*}
    \left\langle{\bf \delta\, x}_k,B\mathcal{D}^{2\alpha}_{-}{\bf x}_k\right\rangle&=\sum_{i=1}^{s+1}b_i\left\langle{ \delta x}_k^i,(\mathcal{D}^{2\alpha}_{-}{\bf x}_k)^i\right\rangle\\\left\langle\mathcal{D}^{2\alpha}_{+}(B{\bf  y})_k,{\bf \delta\, x}_k\right\rangle&=\sum_{i=1}^{s+1}\left\langle(\mathcal{D}^{2\alpha}_{+}(B{\bf  y})_k)^i,{  \delta  x}_k^i\right\rangle.
\end{align*}
Hence
\[\delta\Lf(x_\dis^{},y_\dis^{})=-\rho\,h\sum_{k=0}^{N-1}\sum_{i=1}^{s+1}\left(b_i(\mathcal{D}^{2\alpha}_{-}{\bf x}_k)^i +\mathcal{D}^{2\alpha}_{+}(B{\bf  y})_k)^i\right)\cdot \delta  x_k^i.\]
Combine $\delta\Lc_\dis(x_\dis^{},y_\dis^{})+\delta\Lf(x_\dis^{},y_\dis^{})$, the stationarity implies
\[\delta\mathcal{L}_\dis(x_\dis^{},y_\dis^{})=\sum_{k=0}^{N-1}\sum_{i=1}^{s+1}\lp D_iL_\dis(\mathbf x_k)+D_iL_\dis(\mathbf  y_k)-h\rho\left(b_i(\mathcal{D}^{2\alpha}_{-}{\bf x}_k)^i +\mathcal{D}^{2\alpha}_{+}(B{\bf  y})_k)^i\right)\rp\cdot \delta  x_k^i=0.\]
At the internal stages ($i=2,\ldots,s$),  the variations $\delta x_k^i$  are arbitrary for $k=0,\ldots,N-1$, and thus the stationarity yields the interior-stage equations~\eqref{FinalRKi} and~\eqref{FinalRKyi}.
At the main nodes ($i=1$ and $i=s+1$) with the continuity condition $ x_{k-1}^{s+1}= x_{k}^{1}$ gives $\delta x_{k-1}^{s+1}=\delta x_{k}^{1}$ (and the same for $y_k^i$) yield for all $k=1,\ldots,N-1$
\begin{multline*}
   \Big( D_{s+1}L_\dis(\mathbf  x_{k-1})+D_{s+1}L_\dis(\mathbf  y_{k-1})+ D_{1}L_\dis(\mathbf  x_{k})+D_{s}L_\dis(\mathbf  y_{k})\\-h\rho\left(b_i(\mathcal{D}^{2\alpha}_{-}{\bf x}_k)^1 +\mathcal{D}^{2\alpha}_{+}(B{\bf  y})_k)^1+b_{s+1}(\mathcal{D}^{2\alpha}_{-}{\bf x}_{k-1})^{s+1} +\mathcal{D}^{2\alpha}_{+}(B{\bf  y})_{k-1})^{s+1}\right)\Big)\cdot \delta  x_k^i=0,
\end{multline*}
where the only boundary conditions are $\delta x_0^1=0$ and $\delta x_{N-1}^{s+1}:=\delta x_{N}^{1}=0$. Now, given arbitrary  variation $\delta x_k^{i}$ for $k=1,\ldots,N-1$, we obtain the nodal equations~\eqref{FinalRK1} and~\eqref{FinalRKy1},  and the theorem follows.}
\end{proof}

The FDEL equations~\eqref{FinalRK1} and~\eqref{FinalRKi}, corresponding to the $x$-part, define a discrete iteration scheme for the fractional dynamics~\eqref{ContFracDamp:a}. The general FDEL scheme with RKCQ is summarized in Algorithm~\ref{alg:FractionalAlgorithm}. As the system is nonlinear, each step is typically solved using a numerical root-finding method, such as a Newton method, to compute the $s$ unknowns $(x_k^1,\ldots,x_k^{s+1}=x_{k+1}^1)$. 

\begin{algorithm}{}
\begin{algorithmic}[1]
\State {\bf Initial data}: $N,\, h,\,\alpha,\,W_n^{(\alpha)},\, x_0^1,\, \m_{x_0}.$
 \State {\bf Solve for} $x_0^2, \ldots,x_0^{s+1}$ {\bf from} 
\[
\begin{split}
\m_{x_0}=&-D_1L_\dis(x_0^1, \ldots,x_0^{s+1})+\rho hb_1\left[\mathcal{D}^{2\alpha}_{-}{\bf x}_k\right]^1,\\
0=&\phantom{-} D_iL_\dis(x_0^1, \ldots.,x_0^{s+1})-\rho hb_i\left[\mathcal{D}^{2\alpha}_{-}{\bf x}_k\right]^i,\quad \forall\,i=2, \ldots,s.
\end{split}
\]
\State {\bf Initial points:} $x_0^1, \ldots,x_0^{s+1}=x_1^1$
    \For {$k= 1: N-1$} 
    
\State  {\bf Solve for} $x_{k}^2, \ldots,x_k^{s+1}=x_{k+1}^1$ {\bf from} 
\[
\begin{split}
0=&\,\,D_{s+1}L_\dis(x_{k-1}^1, \ldots,x_{k-1}^{s+1})+ D_1L_\dis(x_k^1, \ldots,x_k^{s+1})-\rho h \left(b_1\left[\mathcal{D}^{2\alpha}_{-}{\bf x}_k\right]^1+b_{s+1}\left[\mathcal{D}^{2\alpha}_{-}{\bf x}_{k-1}\right]^{s+1}\right),\\
0=&\,\,D_iL_\dis(x_k^1, \ldots,x_k^{s+1})-\rho h b_i\left[\mathcal{D}^{2\alpha}_{-}{\bf x}_k\right]^i,\quad \forall\,i=2, \ldots,s.
\end{split}
\]
    \EndFor
    \State  {\bf Output:} $(x_1^{i}, \ldots,x_{N-1}^{i}),\quad i=1, \ldots,s+1.$
\end{algorithmic}
\caption{Fractional Variational Integrator (FVI)  with RKCQ}\label{alg:FractionalAlgorithm}
\end{algorithm}

The discrete trajectory $\{\mathbf{x}_k\}_{0:N}$ is generated by solving the FDEL equations~\eqref{FinalRK1}-\eqref{FinalRKi} through Algorithm~\ref{alg:FractionalAlgorithm}. Once this trajectory has been computed, the associated discrete conjugate momenta at the main nodes can be defined via the discrete Legendre transform (see~\cite{MaWe,SinaVerm,SinaSaake} for the general theory in the variational integrator setting). Analogously to the classical  case, we define the discrete momenta as
\begin{align}\label{eq:xpart_frac_hamiltonian_sys2}
\begin{split}
        \m_k^-&=-D_1L_\dis(x_k,x_{k+1})+\frac{\rho h}{2} \left[\mathcal{D}^{\beta}_{-}{\bf x}_k\right]^1\\
   \m_{k+1}^+&=\phantom{-}D_2L_\dis(x_k,x_{k+1})-\frac{\rho h}{2} \left[\mathcal{D}^{\beta}_{-}{\bf x}_k\right]^2.
\end{split}
\end{align}
In terms of these quantities, equation~\eqref{FinalRK1} takes the form of the matching condition
$$\m_k^+=\m_k^-.$$
Letting 
$\m_k$ denote this common value, we obtain a discrete phase-space solution 
$\{(x_k,\m_k)\}_{k=0,\ldots,N}$ defined on the main nodes.
\section{Convergence Analysis and Numerical Results}\label{sec:numericalresult}
In this section, we carry out both a theoretical and a numerical investigation of the convergence properties of the fractional variational integrators based on RKCQ. On the theoretical side, we provide a rigorous analysis of the convergence order for the scheme generated by the 2-stage Lobatto~IIIC method. For the higher-stage variants, where a complete theoretical treatment becomes substantially more intricate, we support the analysis with systematic numerical experiments that empirically verify the anticipated convergence rates. In addition, we examine whether the proposed schemes can accurately capture the energy decay behaviour inherent in fractionally damped systems. For these investigations, we consider two models. The first is a two-dimensional system of coupled harmonic oscillators subject to linear damping. The second is a class of Bagley–Torvik equations, which provides a scalar but structurally rich test case involving fractional derivatives of different orders.

\subsection{Order of Convergence} In this section, we let 
$\beta=2\alpha$ for simplicity of notation and write $D_-^{\beta}x(t)=D_-^{2\alpha}x(t)$. For the mechanical Lagrangian $L(q,\dot q)=\frac{1}{2}\dot q^\top \dot q -U(q)$,  the restricted  Hamiltonian system associated to the FEL equation~\eqref{ContFracDamp:a} is defined as
\begin{align}\label{eq:xpart_frac_hamiltonian_sys}
\begin{split}
\begin{dcases}
\dot x &= \phantom{-}\m\\
\dot 	\m &= -\nabla U(x)-\rho  D_{-}^{\beta}x. 
\end{dcases}
\end{split}
\end{align}
The discrete Lagrangian associated to the $2$-stage Lobatto~IIIC ($c_1=0,\,c_2=1,\, p=2,\, q=1$) reads
\begin{align*}
    L_\dis(x_k,x_{k+1})&=\frac{h}{2}L\left(x_k,\frac{x_{k+1}-x_k}{h}\right)+\frac{h}{2}L\left(x_{k+1},\frac{x_{k+1}-x_k}{h}\right)\\
&=\frac{1}{2h}\left(x_{k+1}-x_k\right)^\top \left(x_{k+1}-x_k\right)-\frac{h}{2} U(x_k)-\frac{h}{2} U(x_{k+1}).
\end{align*}
Now, we want to write the position-momentum form for~\eqref{ContFracDamp:a}. Computing explicitly  the partial derivatives of $L_\dis(x_k,x_{k+1})$ and plugging them into~\eqref{eq:xpart_frac_hamiltonian_sys2}   yields
   \begin{align*}
        \m_k^-&=\frac{1}{h}\left(x_{k+1}-x_k\right)+\frac{h}{2}\nabla U(x_k) +\frac{\rho h}{2} \left[\mathcal{D}^{\beta}_{-}{\bf x}_k\right]^1 \\
   \m_{k+1}^+&=\frac{1}{h}\left(x_{k+1}-x_k\right)-\frac{h}{2}\nabla U(x_{k+1}) -\frac{\rho h}{2} \left[\mathcal{D}^{\beta}_{-}{\bf x}_k\right]^2.
   \end{align*}
 Let $\m_k$ be the common value from the matching condition $\m_k^+ = \m_k^-$. 
Solving the previous two equations with respect to $x_{k+1}$ and $\m_{k+1}$, respectively, we obtain
\begin{subequations}\label{eq:qpform}
    \begin{align}
    x_{k+1} &= x_k + h\!\left(\m_k - \frac{h}{2}\nabla U(x_k) - \frac{\rho h}{2}\left[\mathcal D_{-}^{\beta} \mathbf x_k\right]^1\right).\label{eq:Gmumentum1}\\
  \m_{k+1} &= \m_k - \frac{h}{2}\bigl(\nabla U(x_k) + \nabla U(x_{k+1})\bigr) - \frac{\rho h}{2}\Big(\big[\mathcal D_{-}^{\beta} \mathbf x_k\big]^1 + \big[\mathcal D_{-}^{\beta} \mathbf x_k\big]^2\Big).\label{eq:Gmumentum2}
    \end{align}
\end{subequations} 
Let $\Phi_h:(x_k,\m_k)\mapsto (x_{k+1},\m_{k+1})$ be the associated one-step flow, and let $\varphi_t:(x,\m)\mapsto (x(t),\m(t))$ be the exact flow corresponding to the continuous system~\eqref{eq:xpart_frac_hamiltonian_sys}. We want to study the local truncation error of the method
\begin{equation}\label{eq:local_err}
    \|\varphi_h(x(t_k),\m(t_k))-\Phi_h(x(t_k),\m(t_k))\|:=\|(x_{k+1},\m_{k+1})-(x_{k+1}^h,\m_{k+1}^h)\|.
\end{equation}
Following Theorem~\ref{thm:RKCQ_convergence} and Remark~\ref{rem:full-order}, for $K^{\beta}(s) = s^{\beta}$ with $0 < \beta < 1$, the
 error satisfies
\begin{equation}\label{eq:stage_frac}
 \varepsilon_k^i := \mathcal{D}_-^{\beta}x(t_k^i)
  - D_{-}^{\beta} x(t_k^i)
  = \mathcal{O}\!\left(h^{q+1-\beta}\right)= \mathcal{O}\!\left(h^{2-\beta}\right), \qquad i = 1,2.
\end{equation}
With these ingredients, we can establish the following convergence order result.
\begin{theorem}\label{thm:theoretical_convergence}
Let $\beta\in (0,1)$  and consider the RKCQ associated to the $2$-stage Lobatto~IIIC under the assumptions of Theorem~\ref{thm:RKCQ_convergence}. Then, the corresponding local  errors~\eqref{eq:local_err} to the FVI scheme~\eqref{eq:qpform} with respect to the continuous system~\eqref{eq:xpart_frac_hamiltonian_sys} satisfy
 \begin{equation}\label{eq:global-general}
        \|x_{k+1}^h - x(t_{k+1})\|=  \mathcal O\!\left(h^{\min(3,\, 4-\beta)}\right),\qquad  \|\m_{k+1}^h - \m(t_{k+1})\| =\mathcal O\!\left(h^{\min(3,\, 3-\beta)}\right).
    \end{equation}
In particular, for $\beta \leq 1$, the method is second-order accurate.
\end{theorem}

\begin{proof}
Let $(x(t),\m(t))$ be the exact solution of~\eqref{eq:xpart_frac_hamiltonian_sys}. For a fixed $t_k$, we set $x(t_k):=x_k$ and $\m (t_k):=\m_k$. The Taylor series expansion of $x(t_{k+1})$ and $\m(t_{k+1})$ gives
\begin{align}
    x(t_{k+1}) &= x_k + h\m_k + \frac{h^2}{2}\dot{\m } + \mathcal O(h^3) \nonumber \\
    &= x_k + h\m_k - \frac{h^2}{2}\nabla U(x_k) - \frac{h^2\rho}{2} D_{-}^{\beta} x(t_k) + \mathcal O(h^3). \label{eq:xex-general}\\
\m (t_{k+1}) &= \m_k + h\dot \m_k + \frac{h^2}{2}\ddot{\m } + \mathcal O(h^3) \nonumber \\
    &= \m_k -h\Big(\nabla U(x_k)+ \rho  D_{-}^{\beta}x_k \Big) +\frac{h^2}{2}\bigg(-\frac{d}{dt}\nabla U(x_k)- \rho  \frac{d}{dt}D_{-}^{\beta}x_k \bigg)+ \mathcal O(h^3).\label{eq:pex-general}
\end{align}
For the numerical update~\eqref{eq:Gmumentum1}, we obtain
    \begin{equation*}
  x_{k+1}^h = x_k +h \m_k - \frac{h^2}{2}\nabla U(x_k) - \frac{ h^2\rho}{2}\mathcal D_{-}^{\beta}  x_k^1,
    \end{equation*}
where $x_k^1:=x(t_k+c_1h)=x(t_k)$. Using~\eqref{eq:stage_frac},
    \begin{align*}
  x_{k+1}^h &= x_k +h \m_k - \frac{h^2}{2}\nabla U(x_k) - \frac{ h^2\rho}{2}\Big( D_{-}^{\beta}  x_k^1+\mathcal{O}\!\left(h^{2-\beta}\right) \Big)\\
 &= x_k +h \m_k - \frac{h^2}{2}\nabla U(x_k) - \frac{ h^2\rho}{2} D_{-}^{\beta}  x_k^1+\mathcal{O}\!\left(h^{4-\beta}\right). 
    \end{align*}    
Thus, comparing this last expression with~\eqref{eq:xex-general} yields
\begin{equation*}
   \| x_{k+1}^h-x(t_{k+1}) \| = \mathcal O\!\left(h^{4-\beta}\right) + \mathcal O(h^3) = \mathcal O\!\left(h^{\min(3,\, 4-\beta)}\right) . \label{eq:x-local-combined}
\end{equation*}
For the scheme~\eqref{eq:Gmumentum2}, we treat each part separately in~\eqref{eq:Gmumentum2}.  For the conservative part, we have
  \begin{equation}\label{eq:conservative_expand}
  \frac{h}{2}(\nabla U(x_k) + \nabla U(x_{k+1})) =
  h\nabla U(x_k)+ \frac{h^2}{2}\frac{d}{dt}\nabla U(x_{k})+\mathcal{O}(h^3).
    \end{equation}
For the fractional part, at each stage, we write using~\eqref{eq:stage_frac}
  \begin{equation}
\frac{\rho h}{2}\Big(\mathcal D_{-}^{\beta}  x_k^1 +\mathcal D_{-}^{\beta}x_k^2\Big)=\frac{\rho h}{2}\Big(D_{-}^{\beta}  x(t_k) + D_{-}^{\beta}x(t_{k+1})+\mathcal{O}\!\left(h^{2-\beta}\right)\Big),
\label{eq:pnum-general}
\end{equation}
where $D_{-}^{\beta} x(t)$ is the exact fractional derivative evaluated at $t_k$ and $t_{k+1}$. As $D_{-}^{\beta} x(t_{k+1})$ is nonlocal, we cannot use the Taylor expansion for $x(t_{k+1})$ and then  apply $D_{-}^{\beta}$. Instead, we proceed as follows. Once the exact solution $x(\cdot)$ is fixed, $g(t):=D_{-}^{\beta} x(t)$ is a real-valued function of $t$, and its time derivatives can be expressed as higher-order fractional derivatives of $x$ as
$$g'(t)=\frac{d}{dt}D_{-}^{\beta} x(t)=D_{-}^{\beta+1} x(t),\qquad g''(t)=D_{-}^{\beta+2} x(t).$$
This follows from the composition rule
$\frac{d^n}{dt^n}\circ D_{-}^{\beta} = D_{-}^{(\beta+n)}$ for the RL derivative. The Taylor expansion thus takes the form
\begin{equation}\label{eq:frac_Taylor}
D_-^{\beta}x(t_{k+1})
  = D_-^{\beta}x(t_k)
  + h\,D_-^{\beta+1}x(t_k)
  + \frac{h^2}{2}\,D_-^{\beta+2}x(\xi_k),
\end{equation}
for some $\xi_k \in (t_k, t_{k+1})$. Substituting~\eqref{eq:frac_Taylor}
 into~\eqref{eq:pnum-general} gives
  \begin{equation}\label{eq:frac_Taylor_sub}
\frac{\rho h}{2}\Big(\mathcal D_{-}^{\beta}  x_k^1 +\mathcal D_{-}^{\beta}x_k^2\Big)=\rho h D_{-}^{\beta}  x(t_k) + \frac{\rho h^2}{2} D_{-}^{\beta+1}x(t_{k})+\mathcal{O}\!\left(h^{3-\beta}\right).
\end{equation}
Replacing~\eqref{eq:conservative_expand} and~\eqref{eq:frac_Taylor_sub} into~\eqref{eq:Gmumentum2}, we obtain
  \begin{align}
 \m_{k+1}^h &= \m_k  -h\nabla U(x_k)- \frac{h^2}{2}\frac{d}{dt}\nabla U(x_{k})- \rho h D_{-}^{\beta}  x(t_k) - \frac{\rho h^2}{2} D_{-}^{\beta+1}x(t_{k})+\mathcal{O}\!\left(h^{3-\beta}\right).
\end{align}
From this last expression and~\eqref{eq:pex-general}, we obtain
$\|\m_{k+1}^h-\m(t_{k+1})\|=\mathcal{O}\!\left(h^{3-\beta}\right).$
\end{proof}

\begin{remark}\label{rem:theoretical_convergence}
From~\eqref{eq:stage_frac}, each $\varepsilon_k^i$ satisfies $\varepsilon_k^i = \mathcal O(h^{2-\beta})$ for $ i=1,2$, $k\geq 0$ individually. Thus, when substituting into~\eqref{eq:pnum-general}, one would naively expect the leading term to be of the form  $C_1 h^{2-\beta}+C_2h^{2-\beta}=(C_1+C_2) h^{2-\beta}$. However, this estimate does not exclude the possibility of cancellation at leading order, thereby increasing the order of the numerical scheme. 
\end{remark}

\subsection{Two-dimensional Forced Harmonic Oscillator}
We consider the two-dimensional oscillator system  subject to linear damping  (with $\beta=1$),
\begin{subequations}\label{eq:damped-oscillator}
\begin{align}
    \ddot x_1 +\eta_1x_1  =- \rho_1\,D_{-}^{\beta} x_1, \\
   \ddot x_2 +\eta_2x_2   =- \rho_2\,D_{-}^{\beta}x_2.
\end{align}
\end{subequations}
The associated conservative Lagrangian is
$$L(x,\dot x)=\frac{1}{2}\left(\dot  x_1^2+\dot x_2^2\right)-\frac{1}{2} \left(\eta_1x_1^2+\eta_2x_2^2\right),$$ 
with $x=(x_1,x_2)$.
In the simulations, the initial values and the parameters are chosen as  
$$x_1(0)=1,\ x_2(0)=-0.5,\ \dot x_1(0) =0.5,\ \dot x_2(0)=0,\quad \eta=\eta_1=\eta_2=1,\quad \rho=\rho_1=\rho_2=0.05.$$
The numerical trajectories shown in Figure~\ref{fig:forced_oscillator}, produced by Algorithm~\ref{alg:FractionalAlgorithm} with 2-stage Lobatto~IIIC method over $[0,150]$ with time step $h=0.2$, are in excellent agreement with the exact solution.
\begin{figure}[htbp]
\centering
\includegraphics[width=.9\textwidth]{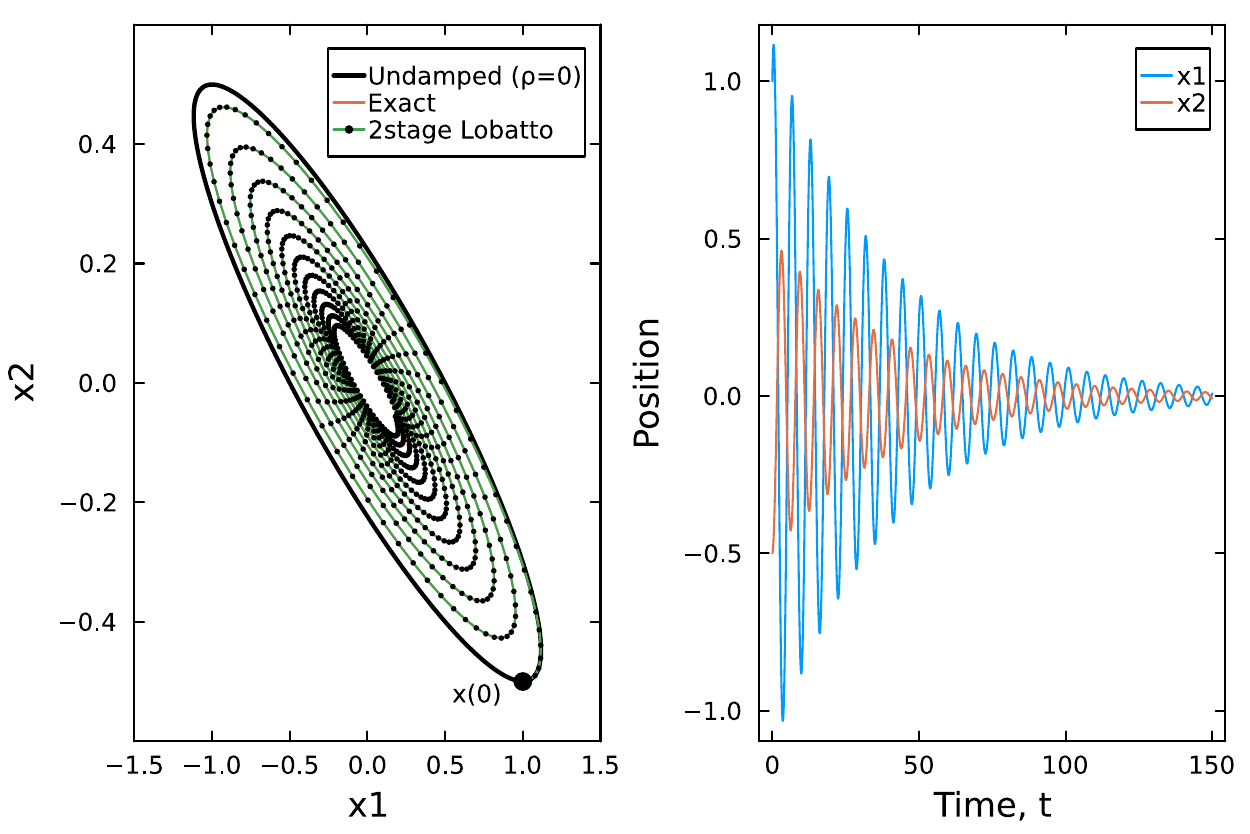}
\caption{Forced oscillator~\eqref{eq:damped-oscillator} $(\beta=1)$. Numerical trajectories obtained by Algorithm~\ref{alg:FractionalAlgorithm} with the 2-stage Lobatto over the interval $[0,150]$ with $h=0.2$.
}
\label{fig:forced_oscillator}
\end{figure}

\noindent The energy of the system is given by the Hamiltonian 
$$E(t)=H(x,\m)=\frac{1}{2}\|\m\|^2+\frac{1}{2}\|x\|^2$$
with $x=(x_1,x_2)$ and $\m=(\m_1,\m_2)$. In Figure~\ref{fig:forced_oscillator_energy} (top left), we observe the numerical energy decay over the interval $[0,120]$ with step size $h=0.2$.  The top right figure displays the relative error of the energy in logarithmic scale, computed  as
\begin{equation}\label{eq:ErrEnergy}
    E_{\text{err}}(t_k)=\frac{E_k-E(t_k)}{\max_{t\in[0,T]}|E(t)|}.
\end{equation}
where $E_k=H(x_k,\m_k)$, i.e.~the energy is evaluated at the approximated solution $x_k=(x_{1,k},x_{2,k})$, $\m_k=(\m_{1,k},\m_{2,k})$, and $E(t_k)=H(x_{\text{ex}}(t_k),\m_{\text{ex}}(t_k))$ is the exact energy at time $t_k$. The relative variations in energy, as can be observed in Figure~\ref{fig:forced_oscillator_energy} (top right),  remain remarkably small, confirming that this physical quantity is very well approximated by the proposed numerical methods. Finally, the bottom figure presents the energy decay   
using FVI with 2-stage Lobatto~IIIC scheme over a long time interval $[0,2000]$ with step size $h=0.5$ and damping factor $\rho=0.0015$. For comparison, the result obtained using the classical RK4 method is also shown. Figure~\ref{fig:forced_oscillator_energy} (bottom) confirms that the proposed method reproduces the energy decay more accurately than the classical RK4 method over long time intervals.

\begin{figure}[htbp]
\centering
\includegraphics[width=.9\textwidth]{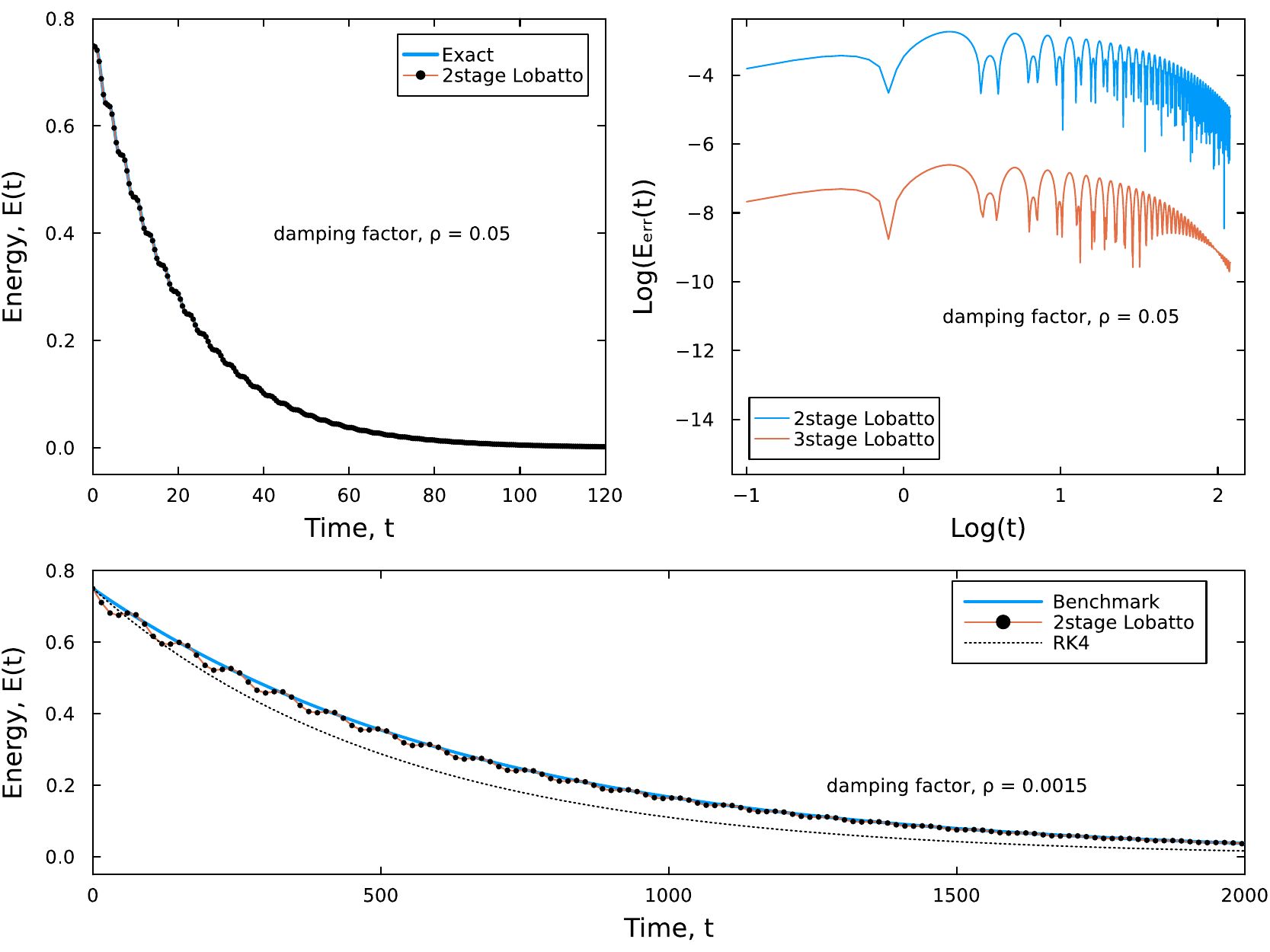}
\caption{Forced Oscillator~\eqref{eq:damped-oscillator}. Energy dissipations and relative errors  of the forced oscillator using FVI with 2-stage Lobatto~IIIC.
 Energy decay  (\textit{top left}) and
energy relative error \eqref{eq:ErrEnergy} (\textit{top right}) over $[0,120]$ with $h=0.2$.  Energy decay (\textit{bottom}) compared with the classical RK4 method over $[0,2000]$ with  $h=0.5$.}
\label{fig:forced_oscillator_energy}
\end{figure}
Now, we test the efficiency of the FVI scheme~\eqref{FinalRK1}-\eqref{FinalRKi} based on  $r$-stage Lobatto~IIIC  with $r=2,3,4$. Let $(x(t) , \m(t))$ be the exact position-momentum solution of the Hamiltonian system associated to~\eqref{eq:damped-oscillator}. The errors at the main nodes are then computed as the global errors
\[\max_{\substack{k\in\{0,\ldots,N\} \\ i\in\{1,2\}}} \left|x_k - x(t_k)\right|,\qquad\max_{\substack{k\in\{0,\ldots,N\} \\ i\in\{1,2\}}} \left|\m_k - \m(t_k)\right|.
 \]
The numerical results in Figure~\ref{fig:FVIcovergence-integer} consistently demonstrate the expected second-order accuracy for the scheme based on the 2-stage Lobatto~IIIC method. The numerically determined orders for the 3- and 4-stage Lobatto~IIIC methods are $4$ and $6$, respectively.  
For an $r$-stage Lobatto~IIIC method with $r=2,3,4$, the proposed scheme achieves order $p=2r-2$. This observation coincides with the superconvergence phenomenon reported by~\cite{SinaVerm} in the context of the discrete Lagrange d'Alembert principle.

\begin{figure}[h]
\centering
\includegraphics[width=.6\textwidth]{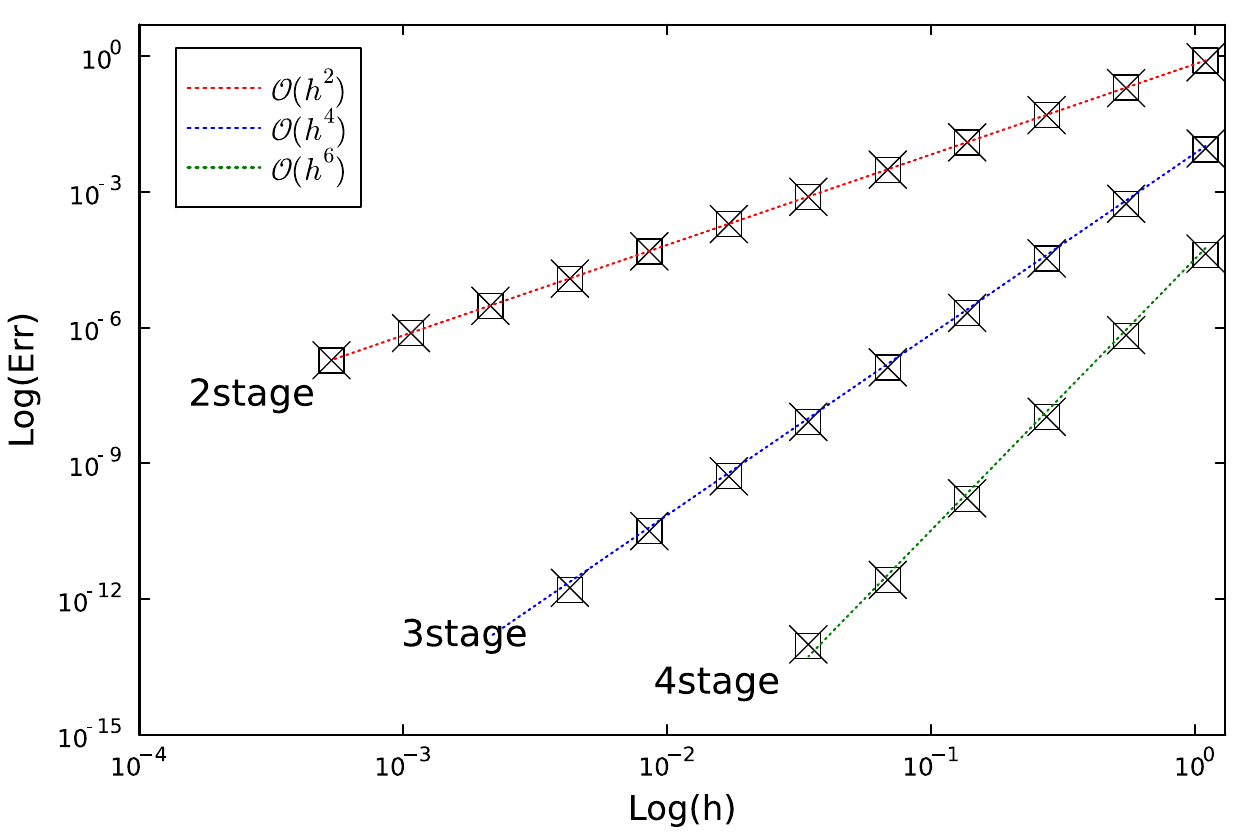}
\caption{Forced Oscillator~\eqref{eq:damped-oscillator}. Numerical verification of the convergence order for position $x$ and momentum $\m$. The results are presented in a logarithm scale plot over the time interval $[0,35]$ with $N=2^i$ for $i=5,\ldots,16$. The results show that the errors for both $x$ and $\m$ are perfectly superimposed, confirming a uniform convergence rate for both variables.}
\label{fig:FVIcovergence-integer}
\end{figure}

\subsection{Bagley-Torvik Equations (BTE)} Consider the fractionally damped equation with forcing:
\begin{equation}
\ddot x + \rho\ D_{-}^{\beta} x +  x =f(t),\quad x(0)=\dot x(0)=0. \label{eq:Torfik1}
\end{equation}
which belongs to the family of Bagley-Torvik equations and can be derived via the restricted variational principle using a time-dependent Lagrangian of the form
$$L(t,x,v)=\frac{v^2}{2}-\frac{x^2}{2}+xf(t).$$
In particular, equations involving $1/2$-order or $3/2$-order fractional derivatives can model systems in which damping depends on frequency, and arise  in the motion of a rigid plate immersed in a viscous fluid and in the behaviour of a gas in a fluid~\cite{Torvik,Igor}. Analytical solutions of such problems are generally difficult to obtain and,  when available, are often expressed in terms of special functions, e.g.,~the Mittag-Leffler function, making their evaluation computationally expensive. We therefore validate the integrator on three complementary
configurations of \eqref{eq:Torfik1}: a standard benchmark with discontinuous forcing, a benchmark with smooth forcing, and a family with a closed-form solution.

We begin with the standard benchmark with discontinuous forcing, which probes the behaviour of
the scheme in the absence of regularity. We take
$$\beta=\frac{3}{2},\quad\rho=0.15,\qquad f(t)=\begin{dcases}
8,\quad 0\leq t\leq 1\\
0,\quad t\geq 1.
\end{dcases}$$
For this problem, a high-accuracy reference solution is obtained via the Talbot contour method for numerical Laplace-transform inversion~\cite{talbot_79,weideman_2006}. The resulting solution serves as a reliable benchmark against which alternative numerical schemes can be validated. In Figure~\ref{fig:fractional_solution_energy} (top left), the numerical solution is plotted over $[1,40]$ with step size $h=0.08$ using the 2-stage  Lobatto~IIIC method. To assess efficiency, we measure the errors in both the configuration and momentum variables by
\begin{equation}\label{max_err}
    \max_{\substack{k\in\{0,\ldots,N\}}} \left|x_k - x(t_k)\right|,\qquad \max_{\substack{k\in\{0,\ldots,N\}}} \left|\m_k - \m(t_k)\right|
\end{equation}
over the time interval $[1,20]$ with  $N = 10\times 2^i,\  i =1,\ldots,8$. 
Although the Talbot method is efficient, its accuracy deteriorates for $t$ close to $0$~\cite{talbot_79}. For this reason, the convergence study uses a Talbot reference solution evaluated away from the origin $t=0$.
As shown in Figure~\ref{fig:fractional_solution_energy} (top right), the convergence rates for position and momentum eventually saturate at first order for the 2-stage Lobatto~IIIC method. This is expected: the order predicted by Theorem~\ref{thm:theoretical_convergence} requires at least  class $C^{3+\beta}$; see Theorem~\ref{thm:RKCQ_convergence}, which is not available here due to the discontinuity of the forcing term, so one would not expect high-order accuracy without sufficient regularity, regardless of the number of stages employed. The energy behaviour is consistent with this. We define the mechanical energy as
$$E(t)=\frac{\dot x(t)^2}{2}+\frac{ x(t)^2}{2}.$$
Taking the time derivative yields $\dot{E}(t) = \dot{x}\,\ddot{x} + x\,\dot{x} = \dot{x}(\ddot{x} + x)$, 
and using equation~\eqref{eq:Torfik1} gives
$$\dot{E}(t) = \dot{x}\,\ddot{x} + x\,\dot{x} = \dot{x}\big(f(t)-\rho\ D_{-}^{\beta} x\big)= \dot{x}f(t)-\rho\dot{x}\ D_{-}^{\beta} x.$$
In Figure~\ref{fig:fractional_solution_energy} (bottom left), the mechanical energy is plotted over $[0,40]$ with step size $h=0.08$ using the 2-stage  Lobatto~IIIC method.  On the interval $[0,1]$, the external forcing  $f(t)=8$.  The term $8\dot x$ represents the power injected by the external force, while the term $\rho\dot x \ D_{-}^{\beta}x$ corresponds to the energy dissipated by the fractional damping. The sharp spike observed near $t=1$ is a consequence of the discontinuity in the forcing term. Since the motion starts from zero initial conditions, the forcing dominates at early times. 
Therefore, the energy increases in that regime. For $t>1$,  the forcing vanishes and $\dot E=-\rho\dot xD_-^{3/2} x$. As shown in the figure, the energy then decreases monotonically, up to small oscillatory modulations, and tends to zero as time increases.
Restricting to $t \geq 1$ removes the singular region from view, and on this interval the numerical solution obtained with $h=0.08$ agrees closely with the benchmark. For the coarser step size $h = 0.16$, small deviations in the  energy from the benchmark are visible in Figure~\ref{fig:fractional_solution_energy} (bottom right). To highlight the role of regularity in the convergence behaviour, we now consider the same equation with a smooth forcing term.
\begin{figure}[h]
\centering
\includegraphics[width=0.9\textwidth]{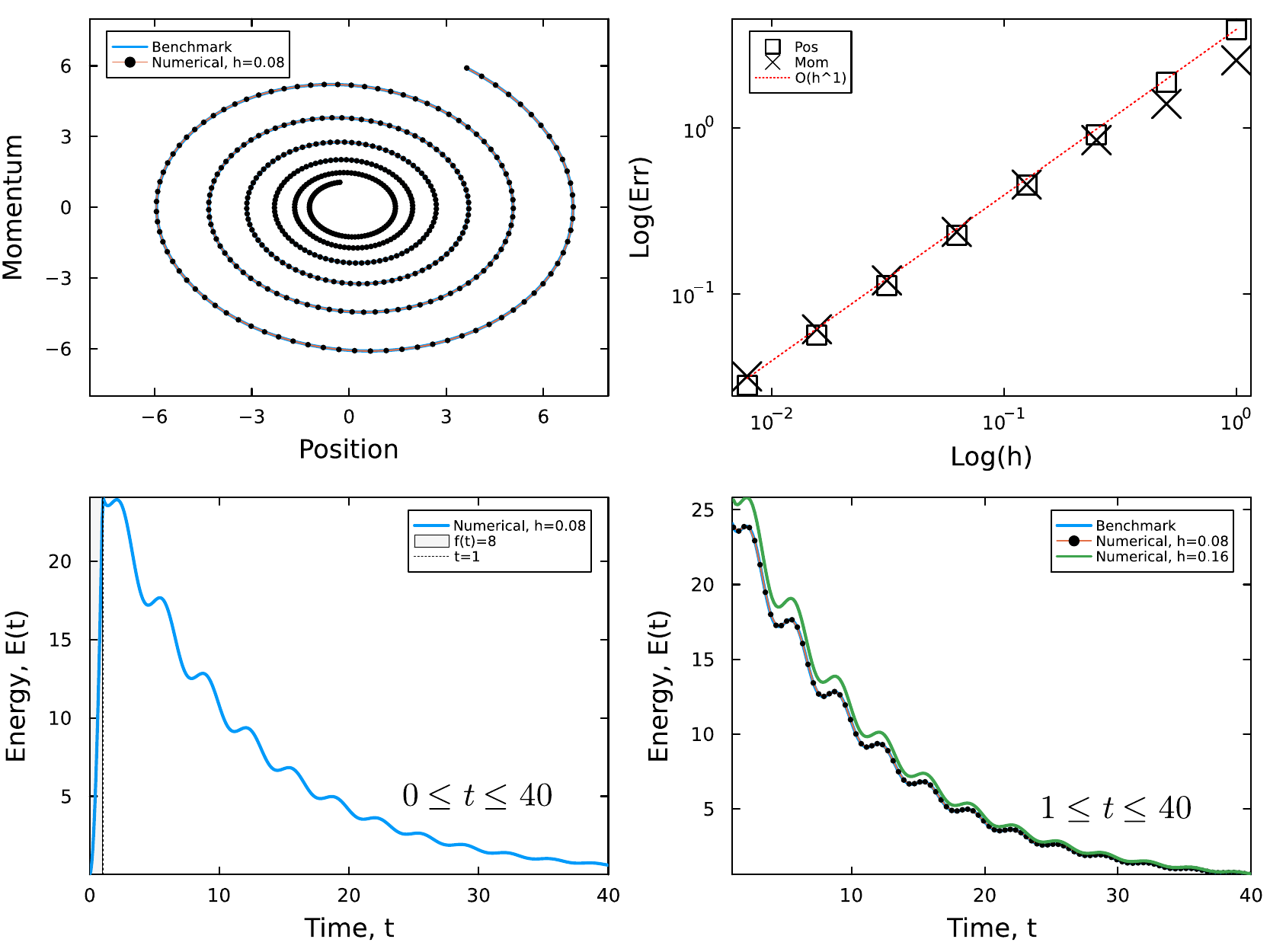}
\caption{BTE with nonsmooth forcing. Comparison between the benchmark and the numerical results for $\beta = 3/2$, obtained using Algorithm~\ref{alg:FractionalAlgorithm} with 2-stage Lobatto~IIIC method: numerical and exact solutions  (\textit{top left}), numerical verification of the convergence order for position $x$ and momentum $\m$ (\textit{top right}) over the time interval $[0,20]$ with  $N = 10\times 2^i,\  i =1,\ldots,8$. Energy evolution (\textit{bottom}).}
\label{fig:fractional_solution_energy}
\end{figure}

To isolate the role of regularity we next consider the same equation with a smooth forcing term and $0<\beta<1$
\begin{equation}
\ddot x + \rho\ D_{-}^{\beta} x + x =\sin(t),\quad x(0)=\dot x(0)=0.  \label{eq:Torfik2}
\end{equation}
Let us pick $\beta= 0.85$ with the same damping factor $\rho =0.15$. We evaluate the motion over the time interval $[0,40]$ using step size $h=0.25$. As shown in Figure~\ref{fig:energy_rel_errors_step_0.25_085} (left), the mechanical energy obtained with the 2-stage Lobatto~IIIC method evolves smoothly, unlike the discontinuous case, the numerical solution  tracks the benchmark over the full interval. Figure~\ref{fig:energy_rel_errors_step_0.25_085} 
  (right) presents the energy errors on a logarithmic scale for the 2-, 3-, and
  4-stage methods. A clear reduction in the error is observed as the number of stages increases, confirming that for smooth forcing data, higher-order methods achieve the expected improvement in accuracy.
\begin{figure}[h]
\centering
\includegraphics[width=0.8\textwidth]{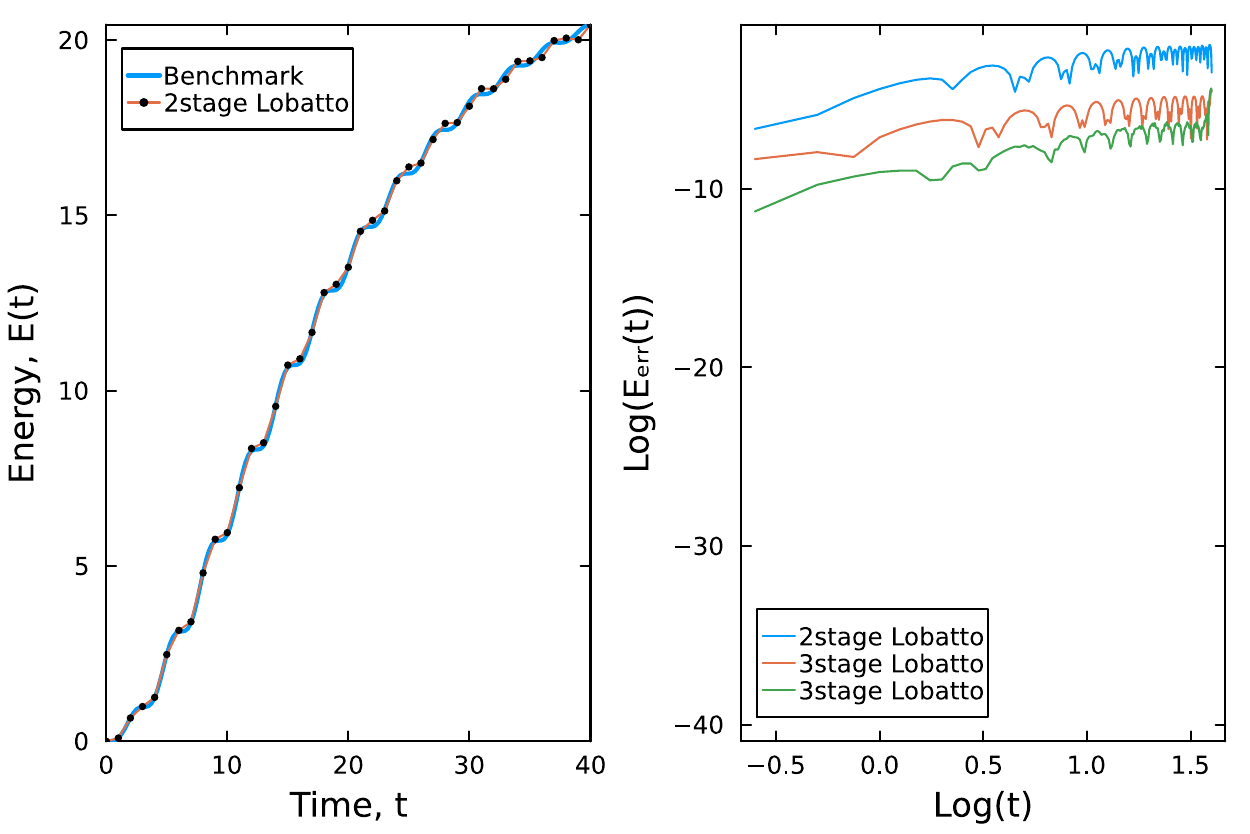}
\caption{BTE with smooth forcing. Numerical results for $\beta = 0.85$, obtained by Lobatto IIIC methods: (\textit{left}) mechanical energy $E(t)$ over $[0, 40]$ computed with the 2‑stage method, compared with the benchmark;  (\textit{right}) energy errors \eqref{eq:ErrEnergy} for the 2‑, 3‑, and 4‑stage methods on a logarithm scale.}
\label{fig:energy_rel_errors_step_0.25_085}
\end{figure}
Figure~\ref{fig:convergence_Benchmark_sin_forcing} measures the errors measures the
position/momentum errors~\eqref{max_err} over $[0,20]$ with $N=2^i$ for $i=4,\ldots,12$, for  $\beta\in \{0.2,\, 0.5,\, 0.75,\, 0.9\}$; the 2-stage scheme attains second order for both the position $x$ and the momentum $\m$ (the latter optimal due to the cancellation effect as in Remark~\ref{rem:theoretical_convergence}), the numerically observed order $4$ for the 3-stage scheme, and and an order slightly below $6$ the 4-stage scheme, which is likely due to the accumulation of round-off and approximation errors on very fine discretizations.
The theoretical rates are thus fully realized once forcing and solution are sufficiently regular.
\begin{figure}[h]
\centering
\includegraphics[width=0.9\textwidth]{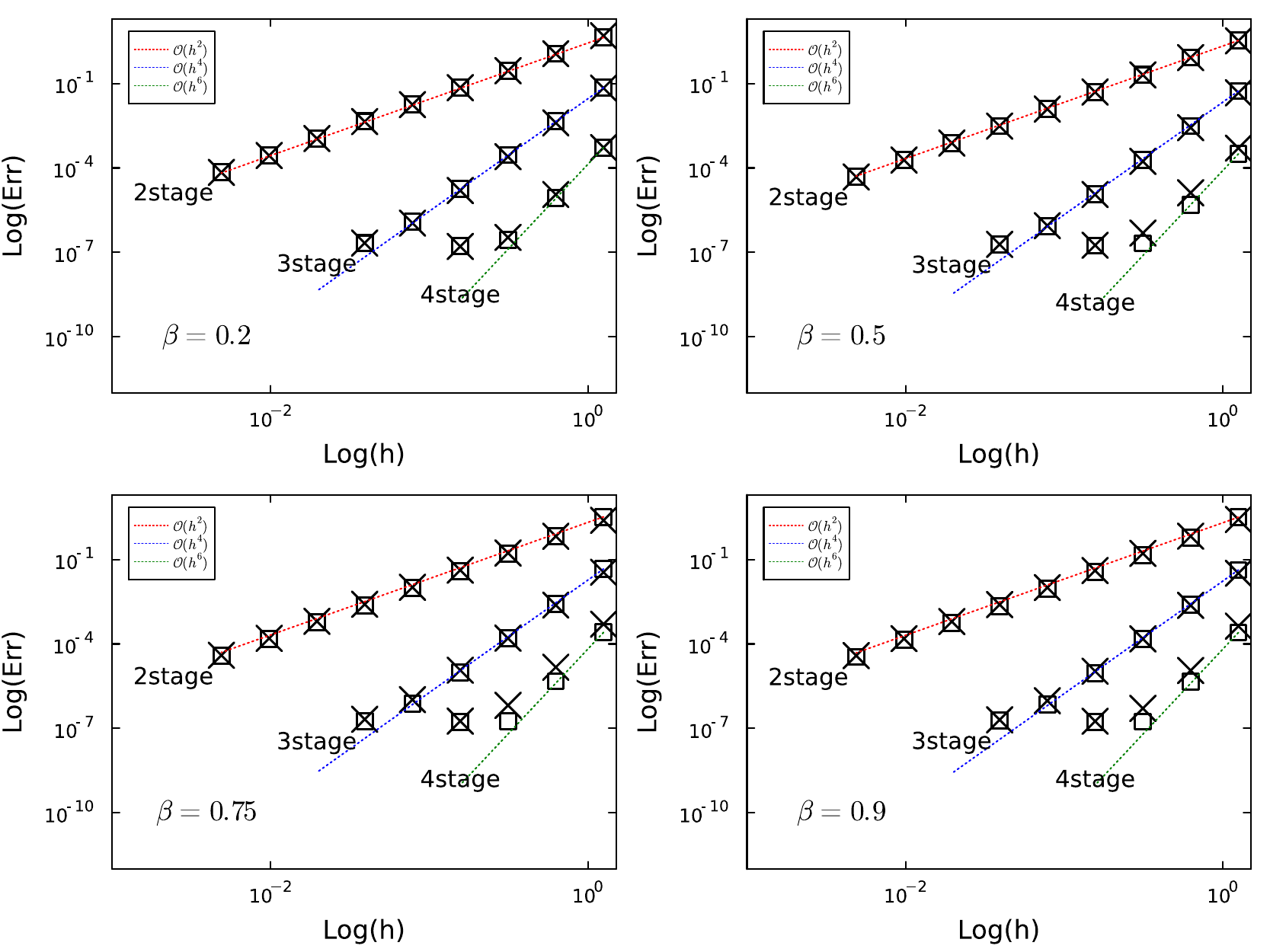}
\caption{BTE with smooth forcing. Numerical verification of the convergence order for position $x$ and momentum $\m$ for different values of $\beta$. The results are presented in a logarithm scale plot over the time interval $[0,20]$ with $N=2^i$ for $i=4,\ldots,12$.}
\label{fig:convergence_Benchmark_sin_forcing}
\end{figure}

Finally, to compare against an exact solution we choose the forcing
\begin{equation}\label{eq:Torfik3}
 f_{\ell,\beta}(t) = (\ell+\beta)(\ell+\beta-1)t^{\ell+\beta-2} + t^{\ell+\beta} + \frac{\Gamma(\ell+\beta+1)}{\Gamma(\ell+1)}  t^\ell\quad \text{with} \quad \rho=1,
\end{equation}
for which $x_{\ell,\beta}(t) = t^{\ell+\beta}$ solves~\eqref{eq:Torfik1} exactly. The errors~\eqref{max_err} are measured on $[0,5]$ with steps $N = 2^i,\ i=2,\ldots,10$, and the orders are shown in Figure~\ref{fig:covergence_all} for various values of $\ell$ and $\beta$. The 2-stage scheme confirms (top plots) the second-order convergence rate of Theorem~\ref{thm:theoretical_convergence} and Remark~\ref{rem:theoretical_convergence}, the 3-stage scheme (middle plots) attains $p = 2r - 2 = 4$, and the 4-stage scheme (bottom plots) again saturates slightly below sixth order. This behaviour is likely due to the combined effects of floating-point round-off and the increased numerical sensitivity involved in computing the RKCQ weights for higher-stage methods. In particular, the larger matrix operations required (the diagonalization and the FFT of Algorithm~\ref{alg:Fractional_weights_Algorithm}) introduce additional numerical errors that accumulate together with round-off effects on the finest grids. These weights are then employed in Algorithm~\ref{alg:FractionalAlgorithm} to approximate the fractional derivatives, so any loss of accuracy in their computation directly affects the overall convergence behaviour.
\begin{figure}[h]
\centering
\includegraphics[width=0.9\textwidth]{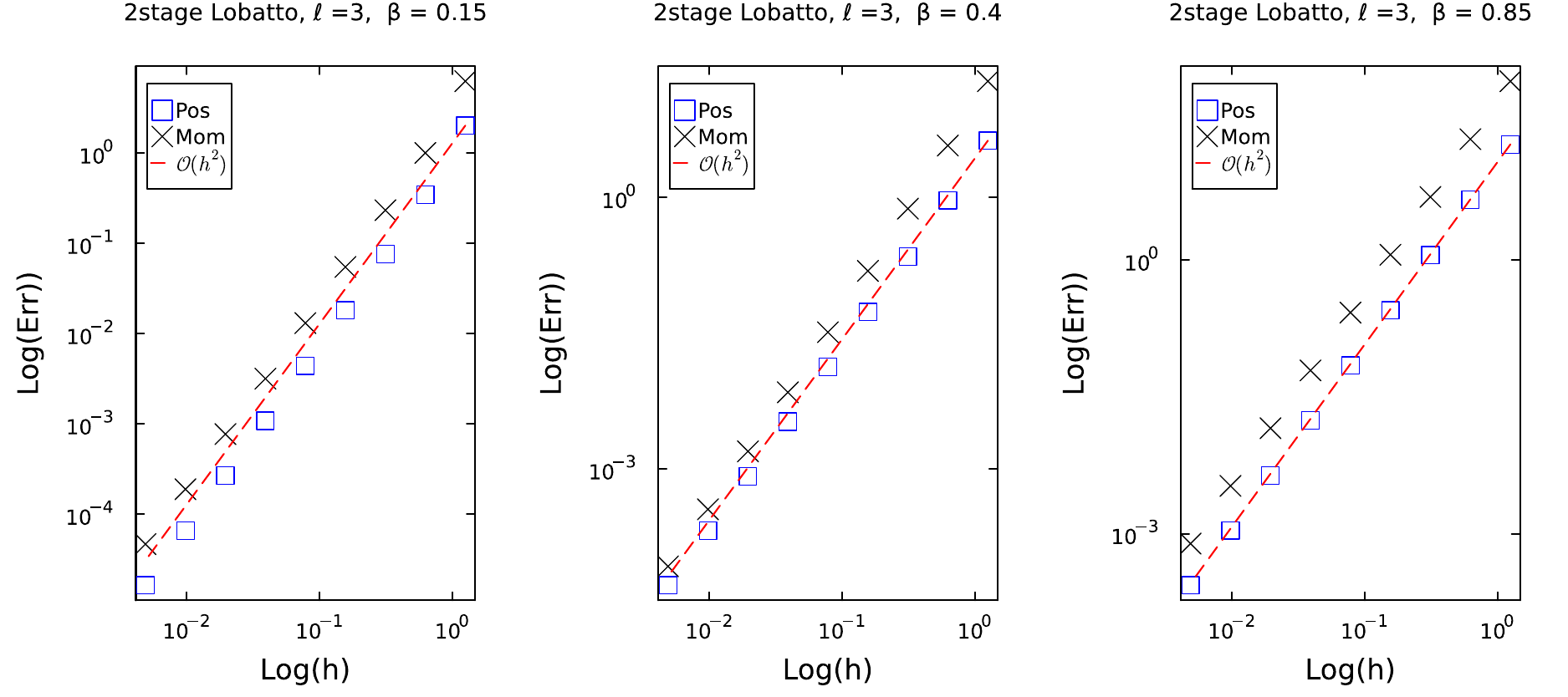}\\
\includegraphics[width=0.9\textwidth]{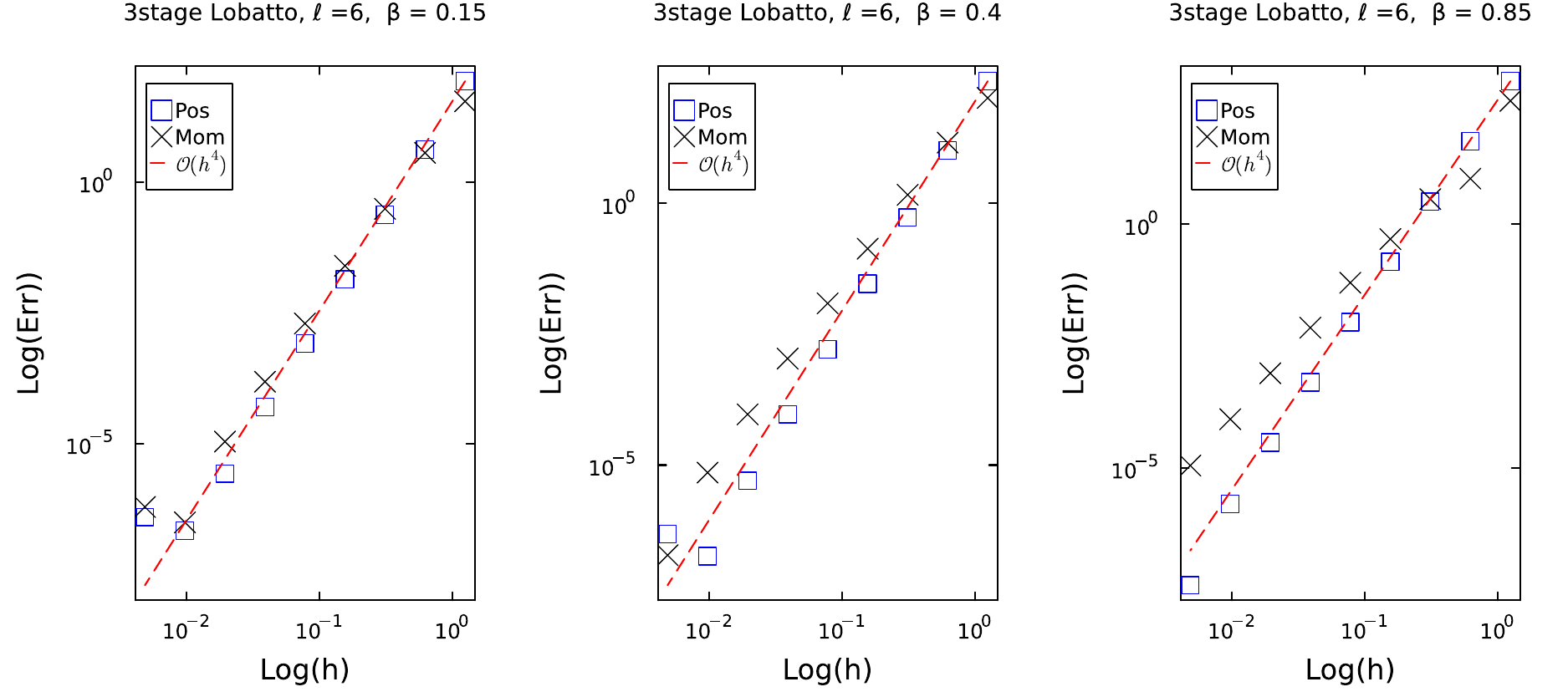}\\
\includegraphics[width=0.9\textwidth]{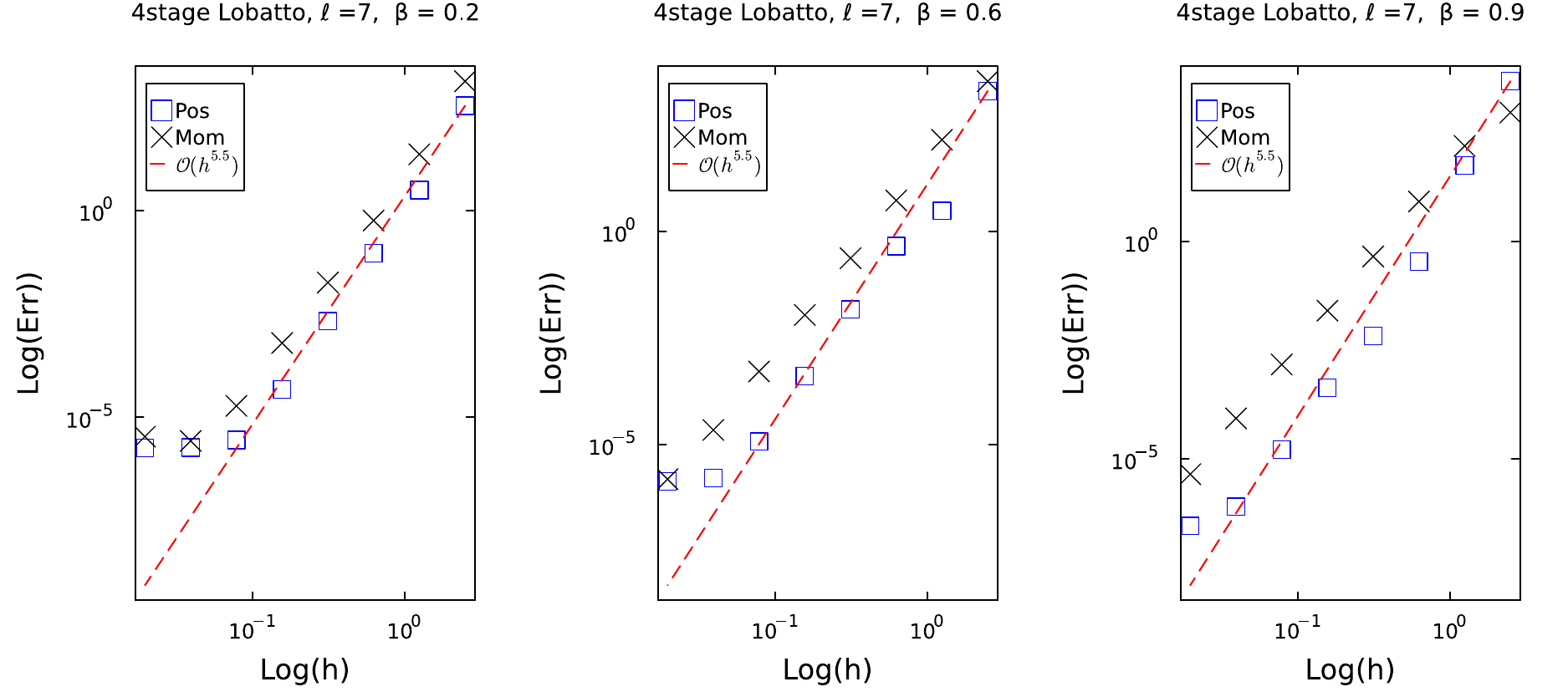}
\caption{BTE with analytical solution~\eqref{eq:Torfik3}. Convergence of position and 
  momentum errors for the FVI associated with the 2-stage Lobatto~IIIC (\textit{top}), 3-stage (\textit{middle}), and 4-stage (\textit{bottom}) Lobatto~IIIC methods over $t\in [0, 5]$ with $N = 2^i$; $i=2,\ldots,10$ for various values of $\ell$ and $\beta$.}
\label{fig:covergence_all}
\end{figure}
\subsection{Note on RKCQ based on Gauss methods}
The analysis of RKCQ based on Gauss methods, when applied to hyperbolic operators, has recently been reported in~\cite{LeMa}. 
\begin{remark}\label{rem:gaussRKCQ}
The definition of RKCQ in~\eqref{RKConQua} is based on evaluating the function \(f\) at the RK nodes \(\{t_k + c_i h\}_{i=1}^r\). On the other hand, the variations \(\delta x_k^i\) appearing in~\eqref{eq:variation_proof} involve not only the internal stages but also the main nodes, namely, the interval endpoints corresponding to \(c_1 = 0\) and \(c_r = 1\). These must coincide with the main-node variations in Lemma~\ref{VariationCommRK}. This is precisely why we employ the Lobatto~IIIC method in the discrete action~\eqref{DiscFracActionHORK}. For Gauss methods, all stages lie strictly in the interior of the interval, that is, \(0 < c_1 < \cdots < c_r < 1\). Consequently, the main-node values are not directly available, and one would need to interpolate the function \(f\).
 \end{remark}
In particular, with the  midpoint   $\setlength{\tabcolsep}{15pt}
\renewcommand{\arraystretch}{1.2}\begin{array}{c|c}
\frac{1}{2}   &  \frac{1}{2}   \\ \hline
     &  1
\end{array},$
the CQ generating function is  given by
$$\gamma_{\text{mid}}(z)=2\frac{1-z}{1+z},$$
which coincides with the generating function of the trapezoidal rule, a method for which convolution quadrature theory is already well established in~\cite{ErSa} and its theoretical order is $2$ at the midpoint stage, see also~\cite{LehSaBook} for more details.
We can define a modified midpoint CQ method (MIDCQ) using a local linear interpolation $\tilde f(t)=f_j+(t-t_j)(f_{j+1}-f_j)/h$ on each subinterval $[t_j,t_{j+1}]$, that is
$$\mathcal{D}_{-}^{\alpha} {f}_k\approx \sum_{j=0}^{k}w_{k-j}^{\beta}\tilde f\left(t_j+h/2\right)=\sum_{j=0}^{k}w_{k-j}^{\beta}\left( \frac{f_j+f_{j+1}}{2}\right),$$
where $w_{k-j}^{\beta}$ are the coefficients of the Taylor expansion of $(\gamma_{\text{mid}}(z)/h)^\beta$. This allows one to approximate the fractional derivatives at $t_k+h/2$ and use the variation at the main nodes which are needed to derive the midpoint scheme.
\begin{algorithm}{}
\begin{algorithmic}[1]
\State {\bf Initial data}: $N,\, h,\,\beta,\,w_n^{\beta},\, x_0,\, \m_{0}.$
\State {\bf Solve for} $x_1$ {\bf from} 
\[
\m_{0}=-D_1L_\dis(x_0,x_1)+ \frac{\rho\, h}{2}w_0^{\beta}\left(\frac{x_1+x_0}{2}\right)
\]
\State {\bf Initial points:} $x_0, x_1$
    \For {$k= 1: N-1$} 

\State  {\bf Solve for} $x_{k+1}$ {\bf from} 
\[
0=D_{2}L_\dis(x_{k-1},x_{k})+ D_1L_\dis(x_k,x_{k+1})- \frac{\rho\, h}{2} \left(\mathcal{D}^{\beta}_{-}{\bf x}_k+\mathcal{D}^{\beta}_{-}{\bf x}_{k-1}\right)
\]
    \EndFor
    \State  {\bf Output:} $(x_2, \ldots,x_{N}).$
\end{algorithmic}
\caption{FVI with MIDCQ}\label{alg:FractionalAlgorithm2}
\end{algorithm}
To test the proposed scheme, we apply Algorithm~\ref{alg:FractionalAlgorithm2} to the BTE with the forcing term and analytical solution defined in~\eqref{eq:Torfik3}. Figure~\ref{fig:covergence_midpoint} shows that  
 the FVI based on MIDCQ, as expected, still achieves full second-order convergence, confirming that its accuracy is not controlled by the value of $\beta$.

\begin{figure}[h]
\centering
\includegraphics[width=1\textwidth]{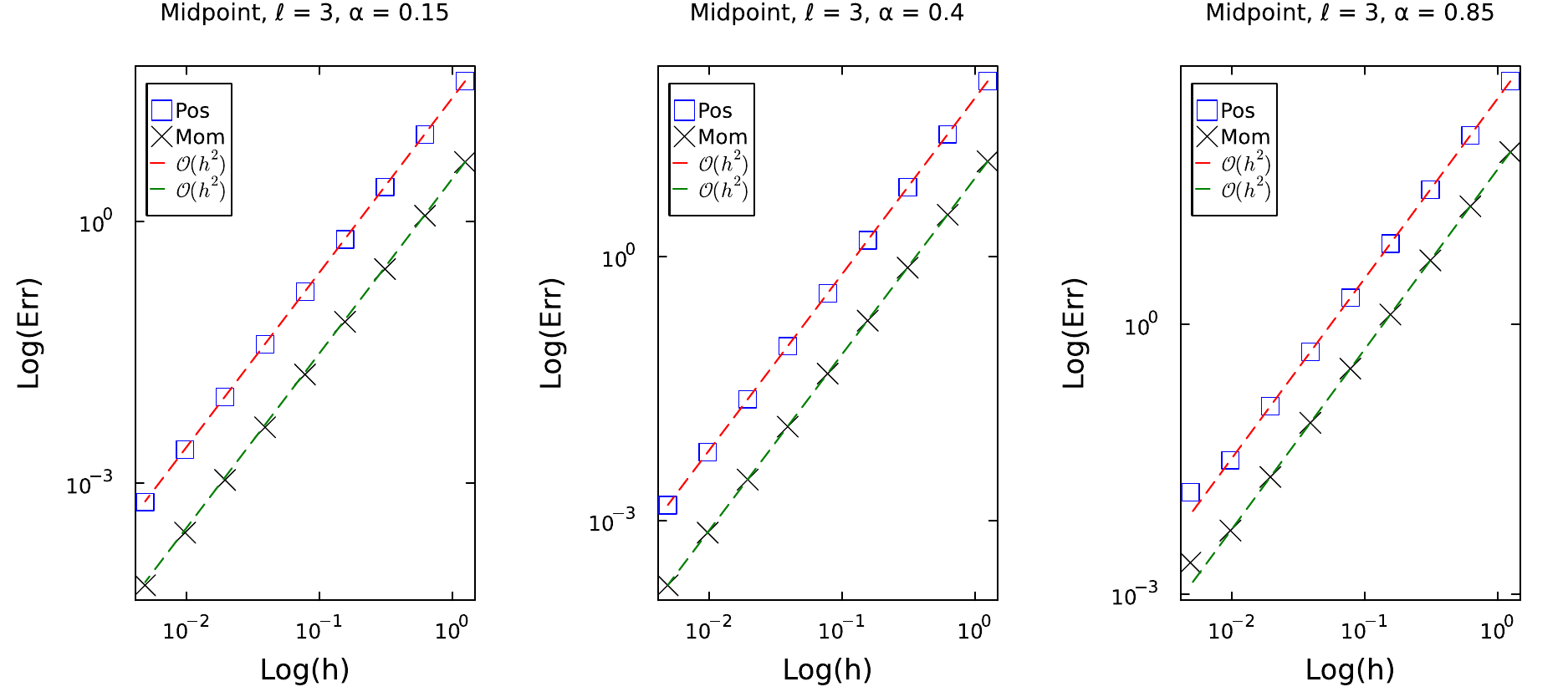}
\caption{BTE with analytical solution~\eqref{eq:Torfik3}. Convergence of position and 
  momentum errors for the FVI associated with the midpoint rule (Algorithm~\ref{alg:FractionalAlgorithm2}) over $t\in [0, 20]$ with $N = 2^i$; $i=4,\ldots,12$ for $\ell=3$ and various values of $\beta$.}
\label{fig:covergence_midpoint}
\end{figure}

\newpage

\section{Conclusion}\label{conclusion}
In this work, we developed a discrete formulation of the restricted Hamilton principle for fractionally damped Lagrangian system, obtained by combining two complementary discretization strategies for the action~\eqref{FracAction}: the conservative part $\Lc$ is approximated by higher-order Galerkin variational methods~\cite{HaLe13,MaWe,SinaSaake}, while the fractional part $\Lf$ is discretized by  Runge-Kutta convolution quadrature (RKCQ) based on Lobatto~IIIC methods~\cite{LuOs,BaLu,BaLo}. The idea is that the internal stages of an RKCQ scheme can be matched, node by node, with the internal nodes used by the Galerkin discretization of the conservative action, so that both contributions to the action are evaluated on a common stage grid $\{t_k + c_i h\}_{i=1}^r$ . Concretely, the paper makes the following contributions
\begin{enumerate}[leftmargin=2.5em]
  \item We constructed a class of fractional variational integrators whose stage
        structure is compatible with the underlying higher-order Galerkin
        discretization of the conservative action (Algorithm~\ref{alg:FractionalAlgorithm}, Theorem~\ref{thm:maintheorem}).
  \item We established the discrete identities required for the variational
        derivation in the convolution-quadrature setting, namely the discrete
        semigroup property (Lemma~\ref{lemma:Semi_group}) and the asymmetric discrete
        integration-by-parts formulas (Lemmas~\ref{ConvPropertiesRK1}--\ref{ConvPropertiesRK2}), which together yield the fractional discrete Euler--Lagrange equations.
  \item We proved, for the $2$-stage Lobatto~IIIC scheme applied to a quadratic
        Lagrangian, a rigorous second-order convergence result for
        $\beta=2\alpha\in(0,1]$ (Theorem~\ref{thm:theoretical_convergence}), and we demonstrated numerically
        fourth- and sixth-order accuracy for the $3$- and $4$-stage schemes. For the integer case $\beta=1$ the observed rate is consistent
        with the variational error theory of~\cite{SinaVerm}.
  \item We introduced a midpoint-based fractional variational integrator (MIDCQ),
        which likewise attains second-order accuracy.
\end{enumerate}
The numerical experiments further show that the proposed integrators reproduce the
expected energy-decay behaviour of the damped dynamics and that their performance
remains stable across a range of fractional orders~$\beta$.

A central advantage of the present RKCQ-based construction over the BDFCQ-based approach of~\cite{HaJiOb} is its full compatibility with the internal stages of the Galerkin variational discretization. Since BDFCQ is a multistep method, it acts only on the main nodal values and does not exploit the internal stage values that carry the higher-order information of the conservative approximation. By contrast, Lobatto~IIIC-based RKCQ naturally incorporates the same stage nodes $\{t_k + c_i h\}_{i=1}^r$
as the Galerkin discrete Lagrangian, so that both the conservative and fractional terms are evaluated on a common grid. This structural compatibility appears to be a key reason why the present approach avoids the saturation effects observed for BDFCQ-based constructions.\\
\noindent\textbf{Limitations.} The computational cost of the proposed method has two main components. For   the conservative part, the result is shown in~\cite{SinaSaake}. For the fractional part, the direct evaluation of the convolution terms requires the storage and processing of the full history of convolution weights $\{W_n^{(\alpha)}\}_{n=0}^{N}$, leading to $\mathcal{O}(N^2)$ computational cost and $\mathcal{O}(N)$ memory usage. Although fast and oblivious convolution techniques~\cite{LuOs} can in principle reduce this complexity to $\mathcal{O}(N \log^2 N)$, their incorporation into the present variational framework has not yet been explored.\\
\noindent\textbf{Future work.} Several directions merit further investigation. On the theoretical side, it would be important to develop a genuine fractional variational error analysis, extending the classical theory~\cite{MaWe} to augmented Lagrangians involving nonlocal operators. 
A second major goal is to establish rigorous higher-order convergence results for the cases 
 $r=3$ and  $r=4$, which are presently supported only by numerical evidence. Another natural extension is the construction of fractional variational integrators based on Gauss-type RKCQ, which would require a modified Galerkin framework due to the absence of boundary nodes; see Remark~\ref{rem:gaussRKCQ}.  An important direction is the rigorous establishment of passivity, or discrete dissipativity, for the proposed integrators. Further applications to physically relevant models, including fractionally damped PDE systems, also appear to be a promising direction.

\bibliographystyle{elsarticle-num.bst}
\bibliography{references}

@article{Agra,
author = {Om P. Agrawal},
title = {Formulation of {E}uler-{L}agrange equations for fractional variational problems},
journal = {J. Math. Anal. Appl.},
volume = {272},
number = {1},
pages = {368--379},
year = {2002}
}

@article{BaLo,
author = {Banjai, Lehel and L{\'{o}}pez-Fern{\'{a}}ndez, M},
title = {Efficient high-order algorithms for fractional integrals and fractional differential equations},
pages = {289--317},
journal = {Numerische Mathematik},
volume = {141},
year = {2019}
}

@article{BaLu,
author = {Banjai, Lehel and Lubich, Christian},
title = {An error analysis of {R}unge-{K}utta convolution quadrature},
pages = {483--496},
journal = {BIT Numerical Mathematics},
volume = {51},
year = {2011}
}

@article{BaLuMar,
author = {Banjai, Lehel and Lubich, Christian and Melenk, Jens Markus},
title = {Runge-{K}utta convolution quadrature for operators arising in wave propagation},
pages = {1--20},
journal = {Numerische Mathematik},
volume = {119},
year = {2011}
}

@incollection{Ca13,
  author       = {Campos, C{\'e}dric M},
  title        = {High order variational integrators: a polynomial approach},
  booktitle    = {Advances in Differential Equations and Applications},
  editor       = {Fernando Casas and Vicente Martínez},
  publisher    = {Springer, Heidelberg},
  year         = {2014},
  pages        = {249--258},
volume ={4},
series={SEMA SIMAI},
}

@article{Ca14,
author = {Cédric M. Campos and Sina Ober-Blöbaum and Emmanuel Trélat},
title = {High order variational integrators in the optimal control of mechanical systems},
journal = {Discrete and Continuous Dynamical Systems},
volume = {35},
number = {9},
pages = {4193--4223},
year = {2015}
}

@article{SinaVerm,
author = {Sina Ober-Blöbaum and Mats Vermeeren},
title = {Superconvergence of {G}alerkin variational integrators},
journal = {IFAC-PapersOnLine},
volume = {54},
number = {19},
pages = {327--333},
year = {2021}
}

@article{ShPaWo,
author = {Harsh Sharma and  Mayuresh Patil and Craig Woolsey},
title = {Energy-preserving variational integrators for forced
{L}agrangian systems},
journal = {Commun Nonlinear Sci Numer Simulat},
volume = {64},
number = {},
pages = {159--177},
year = {2018}
}

@article{SinaWeGaLe,
author = {Sina Ober-Blöbaum and Theresa Wenger and Tobias Gail and Sigrid Leyendecker},
title = {Variational multirate integrators},
journal = {IMA Journal of Numerical Analysis},
volume = {64},
number = {},
pages = {1--40},
year = {2021}
}

@article{WeSinaLe,
author = {Theresa Wenger and Sina Ober-Blöbaum and Sigrid Leyendecker},
title = {Variational integrators of mixed order for dynamical systems with multiple time scales and split potentials},
journal = {ECCOMAS Congress 2016--Proceedings of the 7th European Congress on Computational Methods in Applied Sciences and Engineering},
volume = {},
number = {},
pages = {1818--1831},
year = {2016}
}

@incollection{LeSina,
  author       = {Sigrid Leyendecker and Sina Ober-Blöbaum},
  title        = {A Variational Approach to Multirate Integration for Constrained Systems},
  booktitle    = {Multibody Dynamics},
  editor       = {Jean-Claude Samin and Paul Fisette},
  publisher    = {Springer Science+Business Media Dordrecht},
  year         = {2013},
  pages        = {97--121},
volume ={28},
series={Computational Methods in Applied Sciences},
}

@article{Cresson2,
author = {Jacky Cresson},
title = {Fractional embedding of differential operators and {L}agrangian systems},
journal = {J. Math. Phys.},
 volume = {48},
    number = {3},
    pages = {033504},
    year = {2007},
}

@incollection{Cresson3,
  author       = {Jacky Cresson},
  title        = {Fractional variational embedding and {L}agrangian formulations of dissipative partial differential equations},
  booktitle    = {Fractional Calculus in Analysis, Dynamics and Optimal Control},
  editor       = {},
  publisher    = {Nova Publishers, New York},
  year         = {2013},
  pages        = {65--127},
volume ={},
series={Mathematics Research Developments},
}

@book{TheBook,
   title =     {Fractional Integrals and Derivatives. Theory and Applications},
   author =    {Stefan G. Samko and Anatoly A. Kilbas and Oleg I. Marichev},
   publisher = {Gordon},
   year =      {1993},
   edition =   {1},
}

@book{TheBook2,
   title =     {Theory and Applications of Fractional Differential Equations},
   author =    {Anatoly A. Kilbas and Hari M. Srivastava and Juan J. Trujillo},
   publisher = {Elsevier},
   year =      {2006},
   series =    {North-Holland Mathematics Studies},
   edition =   {1},
   volume =    {204},
}

@article{Riewe,
  title = {Nonconservative {L}agrangian and {H}amiltonian mechanics},
  author = {Fred Riewe},
  journal = {Phys. Rev. E},
  volume = {53},
  number= {2},
  pages = {1890--1899},
  year = {1996},
  publisher = {American Physical Society},
}

@article{Riewe2,
  title = {Mechanics with fractional derivatives},
  author = {Fred Riewe},
  journal = {Phys. Rev. E},
  volume = {55},
  number= {3},
  pages = {3581--3592},
  year = {1997},
}

@article{SinaSaake,
  title = {Construction and analysis of higher order {G}alerkin variational integrators},
  author = {Ober-Blöbaum, Sina and Saake, Nils},
  journal = {Adv. Comput. Math.},
  volume = {41},
  pages = {955--986},
  year = {2015}
}

@article{MaWe, 
author={Marsden, J. E. and West, M.},
title={Discrete mechanics and variational integrators},
volume={10},
journal={Acta Numerica},
publisher={Cambridge University Press},
year={2001}, 
pages={357--514}}

@article{LuOs,
       author = {Lubich, C. and Ostermann, A.},
        title = "{Runge-{K}utta methods for parabolic equations and convolution quadrature}",
      journal = {Mathematics of Computation},
         year = {1993},
       volume = {60},
       number = {201},
        pages = {105--131},
}

@article{Lubich1,
author = {Lubich, C.},
title = {Discretized fractional calculus},
journal = {SIAM Journal on Mathematical Analysis},
volume = {17},
number = {3},
pages = {704--719},
year = {1986}
}

@article{Lubich2,
       author = {Lubich, C},
        title = "{Convolution quadrature and discretized operational calculus I and II}",
      journal = {Numerische Mathematik},
         year = {1988},
       volume = {52},
        pages = {129--145,  413--425},
}

@article{Lubich3,
       author = {Lubich, C},
        title = {On the multistep time discretization of linear initial-boundary value problems and their boundary integral equations},
      journal = {Numerische Mathematik},
         year = {1994},
       volume = {67},
        pages = {365--389},
}

@article{LeMa,
       author = {Lehel Banjai and Matteo Ferrari},
        title = {Runge-{K}utta convolution quadrature based on {G}auss methods},
      journal = {Numerische Mathematik},
         year = {2024},
       volume = {156},
        pages = {719--750},
}

@article{Leok2011,
author = {Leok, Melvin and Shingel, Tatiana},
title = {General techniques for constructing variational integrators},
journal = {Frontiers of Mathematics in China},
volume = {7},
number = {2},
pages = {273--303},
year = {2012},
}

@book{Kai,
   author =    {Kai Diethelm},
   title =     {The Analysis of Fractional Differential Equations. An Application-Oriented Exposition Using Differential Operators of Caputo Type},
   publisher = {Springer-Verlag Berlin Heidelberg},
   year =      {2010},
   series =    {Lecture Notes in Mathematics},
volume ={2004},
   edition =   {1},
}

@book{Kaibook2016,
   author =    {Dumitru Baleanu and Kai Diethelm and Enrico Scalas and Juan J Trujillo},
   title =     {Fractional Calculus: Models and Numerical Methods},
   publisher = {World Scientific},
   year =      {2016},
   series =    {Series on Complexity, Nonlinearity and Chaos},
volume ={3},
   edition =   {2},
}

@article{JiOb1,
title = {A fractional variational approach for modelling dissipative mechanical systems: continuous and discrete settings},
author = {Fernando Jiménez and Sina Ober-Blöbaum},
journal = {IFAC-PapersOnLine},
volume = {51},
number = {3},
pages = {50--55},
year = {2018}
}

@article{JiOb2,
author = {Jiménez, Fernando  and Ober-Blöbaum, Sina},
title = "{Fractional damping through restricted calculus of variations}",
journal = {J. Nonlinear Sci.},
volume = {31},
pages = {46},
year = {2021}
}

@article{HaJiOb,
author = {Hariz Belgacem, Khaled and Jiménez, Fernando  and Ober-Blöbaum, Sina},
title = {Fractional variational integrators based on convolution quadrature},
journal = {J. Nonlinear Sci.},
volume = {35},
number = {2},
pages = {38},
year = {2025}
}

@article{HaLe13,
author = {Hall, James  and Leok, Melvin},
title = {Spectral variational integrators},
journal = {Numerische Mathematik},
volume = {130},
pages = {681--740},
number ={4},
year = {2015}
}

@book{HaWa,
   author =    {Ernst Hairer and Gerhard Wanner},
   title =     {Solving Ordinary Differential Equations II. Stiff and Differential-Algebraic Problems},
   publisher = {Springer Berlin, Heidelberg},
   year =      {2010},
number = {14},
   series =    {Springer Series in Computational Mathematics},
   edition =   {2},
}

@book{LuGeWa,
    AUTHOR = {Hairer, Ernst and Lubich, Christian and Wanner, Gerhard},
     TITLE = {Geometric numerical integration. Structure-preserving algorithms for ordinary differential equations},
    SERIES = {Springer Series in Computational Mathematics},
    VOLUME = {31},
   EDITION = {2},
 PUBLISHER = {Springer-Verlag, Berlin},
      YEAR = {2006},
}

@book{Raff2023,
    AUTHOR = {Raffaele D'Ambrosio},
     TITLE = {Numerical Approximation of Ordinary Differential Problems. From Deterministic to Stochastic Numerical Methods},
    SERIES = {UNITEXT},
    VOLUME = {148},
 PUBLISHER = {Springer Nature Switzerland AG},
      YEAR = {2023},
}

@book{SerFer2016,
    AUTHOR = {Sergio Blanes and Fernando Casas},
     TITLE = {A Concise Introduction to Geometric Numerical Integration},
    SERIES = {Monographs and Research Notes in Mathematics},
    VOLUME = {},
 PUBLISHER = {CRC Press, New York},
      YEAR = {2016},
}

@book{Oldham,
   title = {The Fractional Calculus. Theory and Applications of Differentiation and Integration to Arbitrary Order},
   author =    {Keith B. Oldham and Jerome Spanier},
   publisher = {Academic Press, New York, London},
   volume =    {111},
   year =      {1974},
}

@book{AbMa,
   author =    {Abraham, R. and  Marsden, J. E.},
   title =     {Foundations of Mechanics},
   publisher = {Benjamin/Cummings Publishing Company},
   year =      {1978}
}

@article{Torvik,
    author = {Torvik, P. J. and Bagley, R. L.},
    title = {On the appearance of the fractional derivative in the behavior of real materials},
    journal = {Journal of Applied Mechanics},
    volume = {51},
    number = {2},
    pages = {294--298},
    year = {1984},
}

@book{Igor,
   title =     {Fractional Differential Equations. An Introduction to Fractional Derivatives, Fractional Differential Equations, to Methods of Their Solution and Some of Their Applications},
   author =    {Igor Podlubny},
   publisher = {Academic Press},
   year =      {1998},
   series =    {Mathematics in Science and Engineering},
   edition =   {1st},
   volume =    {198},
}

@article{Bagley,
author = {R. L. Bagley and R. A. Calico},
title = "{Fractional order state equations for the control of viscoelastically
damped structures}",
journal = {J. Guid. Control Dyn.},
volume = {14},
number = {2},
pages = {304--311},
year = {1991},
}

@article{Magin1,
author = {R.L. Magin},
title = "{Fractional calculus in bioengineering, Part 2}",
journal = {Journal of Biomechanics},
volume = {37},
number = {5},
pages = {703--711},
year = {2004},
}

@article{Magin2,
author = {R.L. Magin},
title = "{Fractional calculus models of complex dynamics in biological tissues}",
journal = {J. Comput. Math. Appl.},
volume = {59},
number = {5},
pages = {1586--1593},
year = {2010},
}

@article{Bonilla,
author = {B. Bonilla and M. Rivero and L. Rodríguez-Germá and J.J. Trujillo},
title = "{Fractional differential equations as alternative models to nonlinear differential equations}",
journal = {J. Appl. Math. Comput.},
volume = {187},
number = {1},
pages = {79--88},
year = {2007},
}

@book{Hilfer,
   title =     {Applications of Fractional Calculus in Physics},
   author =    {R. Hilfer},
   publisher = {World Scientific, River Edge, New Jersey},
   year =      {2000},
   series =    {},
   edition =   {},
}

@book{MarRat,
   title =     {Introduction to Mechanics and Symmetry. A Basic Exposition of Classical Mechanical Systems},
   author =    {Jerrold E. Marsden and Tudor S. Ratiu},
   year =      {1999},
   series =    {Texts in Applied Mathematics},
 publisher = {Springer Science+Business Media,  New York},
 volume =    {17},
   edition =   {2},
}

@article{Klimek,
author = {M. Klimek},
title = "{Fractional sequential mechanics--models with symmetric fractional derivative}",
journal = {Czechoslovak Journal of Physics},
volume = {51},
number = {12},
pages = {1348--1354},
year = {2001},
}

@article{Baleanu,
author = {D. Baleanu and S.I. Muslih},
title = "{Lagrangian formulation of classical fields within {R}iemann-{L}iouville
fractional derivatives}",
journal = {Phys. Scripta},
volume = {72},
number = {2--3},
pages = {119--121},
year = {2005},
}

@book{Tarasov,
    author = {V. E. Tarasov},
    title = {Fractional Dynamics. Applications of Fractional Calculus to Dynamics of Particles, Fields and Media},
series = {Nonlinear Physical Science},
    publisher = {Springer Berlin Heidelberg} ,
    year = {2010},
}

@book{BoXuHu,
    author = {Boling Guo and Xueke Pu and Fenghui Huang},
    title = {Fractional Partial Differential Equations and Their Numerical Solutions},
    publisher = {World Scientific Publishing} ,
    year = {2015},
}

@book{HandbookNum,
    author = {George Em Karniadakis},
    title = {Handbook of Fractional Calculus with Applications},
volume ={3},
    publisher = {De Gruyter} ,
    year = {2019},
}

@article{Lidia2015,
author = {Aceto, Lidia and Magherini, Cecilia and Novati, Paolo},
title = {On the construction and properties of $m$-step methods for {FDE}s},
journal = {SIAM Journal on Scientific Computing},
volume = {37},
number = {2},
pages = {A653--A675},
year = {2015},
}

@article{ErSa,
author = {Hasan Eruslu and Francisco Javier Sayas},
title = "{Polynomially bounded error estimates for trapezoidal rule
convolution quadrature}",
journal = {Computers and Mathematics with Applications},
volume = {79},
number = {6},
pages = {1634--1643},
year = {2020},
}

@book{LehSaBook,
    author ={Lehel Banjai and Francisco Javier Sayas} ,
    title ={Integral Equation Methods for Evolutionary PDE. A Convolution Quadrature Approach} ,
series ={Springer Series in Computational Mathematics},
volume ={59},
    publisher = {Springer Nature, Switzerland},
    year = {2022}
}

@article{Bauer,
    author = {P. S. Bauer},
    title = {Dissipative dynamical systems},
    journal = {Proc. Natl. Acad. Sci.},
    year = {1931},
volume ={5},
pages = {311--314}
}

@book{chala,
   title =     {Fractional Calculus with Applications in Mechanics. Vibrations and Diffusion Processes},
   author =    {Teodor M. Atanacković and Stevan Pilipović and Bogoljub Stanković and Dušan Zorica},
   publisher = {Wiley-ISTE},
series ={Mechanical Engineering and Solid Mechanics},
   year =      {2014},
   edition =   {},
}

@article{talbot_79,
    author = {Talbot, A.},
    title = {The Accurate Numerical Inversion of {L}aplace Transforms},
    journal = {IMA Journal of Applied Mathematics},
    volume = {23},
    number = {1},
    pages = {97--120},
    year = {1979},
}

@article{weideman_2006,
author = {Weideman, J. A. C.},
title = {Optimizing {T}albot's Contours for the Inversion of the Laplace Transform},
journal = {SIAM Journal on Numerical Analysis},
volume = {44},
number = {6},
pages = {2342--2362},
year = {2006}
}

@article{Sawar_Hussain_Ayaz_2025, 
title={Secure Image Transmission Using Fractional Variable Order Memristive Hyperchaotic System With Nonlinear Synchronization},
volume={2}, 
number={1}, 
journal={Advances in Analysis and Applied Mathematics}, 
author={Sawar, Shah and Hussain, Sadam and Ayaz, Muhammad}, 
year={2025},
pages={32--43} }
\end{document}